\documentclass[11pt]{article}

\usepackage[sort]{natbib}
\usepackage[usenames]{color}

\usepackage[bookmarks=true,bookmarksnumbered=true,colorlinks=true,allcolors=blue]{hyperref}
\usepackage{amsmath,amssymb,amsfonts,amsthm}
\usepackage{lscape}
\usepackage{arydshln}
\renewcommand*{\baselinestretch}{1.25}

\usepackage{geometry}
\usepackage{indentfirst}

\usepackage{enumitem}

\newtheorem{theorem}{Theorem}[section]
\newtheorem{lemma}{Lemma}[section]
\newtheorem{proposition}{Proposition}[section]
\newtheorem{corollary}{Corollary}[section]
\theoremstyle{definition}
\newtheorem{definition}{Definition}[section]
\newtheorem*{rmk*}{Remark}
\newtheorem{rmk}{Remark}[section]

\renewcommand*\proofname{\upshape{\bfseries{Proof}}}

\DeclareMathOperator{\expectation}{E}

\allowdisplaybreaks

\numberwithin{equation}{section}

\makeatletter
    \renewcommand*{\section}{\@startsection{section}{1}{\z@}%
    {6pt}{3pt}{\reset@font\normalsize\bfseries}}
\makeatother
    
\makeatletter
    \renewcommand*{\subsection}{\@startsection{subsection}{2}{\z@}%
    {3pt}{3pt}{\reset@font\normalsize\mdseries\itshape}}
\makeatother

\makeatletter
    \renewcommand*{\subsubsection}{\@startsection{subsubsection}{3}{\z@}%
    {3pt}{3pt}{\reset@font\normalsize\mdseries\itshape}}
\makeatother

\makeatletter
\def\@listi{\leftmargin\leftmargini
  \topsep=.5\baselineskip 
  \partopsep=0pt \parsep=0pt \itemsep=0pt}
\let\@listI\@listi
\@listi
\def\@listii{\leftmargin\leftmarginii
  \labelwidth\leftmarginii \advance\labelwidth-\labelsep
  \topsep=0pt \partopsep=0pt \parsep=0pt \itemsep=0pt}
\def\@listiii{\leftmargin\leftmarginiii
  \labelwidth\leftmarginiii \advance\labelwidth-\labelsep
  \topsep=0pt \partopsep=0pt \parsep=0pt \itemsep=0pt}
\def\@listiv{\leftmargin\leftmarginiv
  \labelwidth\leftmarginiv \advance\labelwidth-\labelsep
  \topsep=0pt \partopsep=0pt \parsep=0pt \itemsep=0pt}
\makeatother

\makeatletter
\renewenvironment{proof}[1][\proofname]{\par
  \pushQED{\qed}%
  \normalfont \topsep6\p@\@plus6\p@\relax
  \trivlist
  \item[\hskip\labelsep
        \bfseries
    #1\@addpunct{.}]\ignorespaces
}{%
  \popQED\endtrivlist\@endpefalse
}
\makeatother

\newcommand{\labs}{\left|}
\newcommand{\rabs}{\right|}
\newcommand{\lpa}{\left(}
\newcommand{\rpa}{\right)}
\newcommand{\ve}{\varepsilon}

\newcommand{\ol}[1]{\overline{#1}}
\newcommand{\ul}[1]{\underline{#1}}
\newcommand{\wh}[1]{\widehat{#1}}
\newcommand{\wt}[1]{\widetilde{#1}}

\newcommand{\ex}[1]{\expectation\left[#1\right]}
\newcommand{\expe}[1]{\mathrm{E}[#1]}
\newcommand{\mf}[1]{\mathfrak{#1}}
\newcommand{\mcl}[1]{\mathcal{#1}}

\title{Notes on the dimension dependence in high-dimensional central limit theorems for hyperrectangles}
\author{Yuta Koike
\thanks{Mathematics and Informatics Center and Graduate School of Mathematical Sciences, The University of Tokyo, 3-8-1 Komaba, Meguro-ku, Tokyo 153-8914 Japan}
\thanks{CREST, Japan Science and Technology Agency}
}

\begin{document}

\maketitle

\begin{abstract}

Let $X_1,\dots,X_n$ be independent centered random vectors in $\mathbb{R}^d$. 
This paper shows that, even when $d$ may grow with $n$, the probability $P(n^{-1/2}\sum_{i=1}^nX_i\in A)$ can be approximated by its Gaussian analog uniformly in hyperrectangles $A$ in $\mathbb{R}^d$ as $n\to\infty$ under appropriate moment assumptions, as long as $(\log d)^5/n\to0$. 
This improves a result of Chernozhukov, Chetverikov \& Kato [\textit{Ann. Probab.} \textbf{45} (2017) 2309--2353] in terms of the dimension growth condition. When $n^{-1/2}\sum_{i=1}^nX_i$ has a common factor across the components, this condition can be further improved to $(\log d)^3/n\to0$. 
The corresponding bootstrap approximation results are also developed. 
These results serve as a theoretical foundation of simultaneous inference for high-dimensional models. 
%
\vspace{3mm}


\end{abstract}

\section{Introduction}

Let $X=(X_i)_{i=1}^n$ be independent centered random vectors in $\mathbb{R}^d$ and consider the normalized sum:
\[
S_n^X=(S_{n,1}^X,\dots,S_{n,d}^X)^\top:=\frac{1}{\sqrt{n}}\sum_{i=1}^nX_i.
\]
We assume that each coordinate of $S_n^X$ has (at least) a finite second moment and write the covariance matrix as $\mf{C}_n^X:=\expe{S_n^X(S_n^X)^\top}$. The aim of this paper is to approximate $S_n^X$ by its Gaussian analog $Z_n^X$ in law, where $Z^X_n=(Z^X_{n,1},\dots,Z^X_{n,d})^\top$ denotes a $d$-dimensional centered Gaussian vector with covariance matrix $\mf{C}_n^X$. 
When $n$ tends to infinity while $d$ is fixed, such an approximation is commonly formulated as convergence in law. Then, it is merely a consequence of a classical multivariate central limit theorem (CLT) under mild regularity assumptions. Nevertheless, in a high-dimensional setting where $d$ grows as $n\to\infty$, the situation is not as simple as above. In such a setting, it is typical that $Z_n^X$ depends on $n$ and has no limit law as $n\to\infty$, so the standard formulation is no longer meaningful. One possible way to properly formulate the problem is to consider the convergence of some metric between the laws of $S_n^X$ and $Z_n^X$. A typical choice of such a metric is the following one:
\[
\rho_n(\mcl{A}):=\sup_{A\in\mcl{A}}\labs P\lpa S_n^X\in A\rpa-P\lpa Z_n^X\in A\rpa\rabs,
\]
where $\mcl{A}$ is a class of Borel sets in $\mathbb{R}^d$. In this regard, investigation of Lyapunov type bounds for $\rho_n(\mcl{A})$ with explicit dimension dependence has some history in the case that $\mcl{A}$ is the class of all convex Borel sets in $\mathbb{R}^d$, which we write $\mathcal{A}^{\mathrm{co}}$ in the following. 
In particular, under appropriate moment conditions, one can conclude $\rho_n(\mathcal{A}^{\mathrm{co}})\to0$ as $n\to\infty$ if $d^{7/2}/n\to0$ from \citet{Bentkus2005}'s result. 
Meanwhile, it has recently attracted much attention in the probabilistic literature to derive bounds for the Wasserstein distances of order $p\geq1$ between the laws of $S_n^X$ and $Z_n^X$ in high-dimensional settings; see \cite{Zhai2018,Bonis2019,CFP2017,EMZ2020,Fathi2018}, among others. 
As illustrated in \cite[Section 1.1]{Zhai2018}, such a bound can be used to improve the dimension dependence to obtain the convergence $\rho_n(\mathcal{A}^{\mathrm{co}})\to0$ under some situations. 
For example, when each $X_i$ is isotropic and satisfies a Poincar\'e inequality with constant $C$ independent of $n$, we can deduce $\rho_n(\mathcal{A}^{\mathrm{co}})\to0$ as $n\to\infty$ if $d^{3/2}/n\to0$ from \cite[Proposition 1.4]{Zhai2018} and \cite[Theorem 4.1]{CFP2017}. 

As outlined above, one typically requires sub-linear dependence of $d$ on $n$ to get $\rho_n(\mathcal{A}^{\mathrm{co}})\to0$ or the convergences of the Wasserstein distances. In fact, one can easily verify that this is usually \textit{necessary} for getting (at least) the latter convergences. 
Nevertheless, in modern data science, one is often interested in a situation where $d$ is (much) larger than $n$. Recently, the path-breaking work of \citet*{CCK2013} has shown that, if we restrict our attention to the class $\mcl{A}=\mcl{A}^m$ of sets of the form $A=\{x\in\mathbb{R}^d:\max_{j\in J}x_j\leq a\}$ for some $a\in\mathbb{R}$ and $J\subset\{1,\dots,d\}$ ($x_j$ denotes the $j$-th coordinate of $x$), we can deduce $\rho_n(\mathcal{A}^m)\to0$ as $n\to\infty$ under appropriate moment conditions even if $d$ is as large as $e^{Cn^c}$ for some $c,C>0$. 
This type of convergence is indeed enough for many statistical applications in high-dimensional inference such as construction of simultaneous confidence intervals and strong control of the family-wise error rate (FWER) in multiple testing; see \cite{BCCHK2018} for details.  
This result has further been extended in \citet{CCK2017} to the case that $\mcl{A}=\mcl{A}^{\mathrm{re}}$ is the class of all hyperrectangles in $\mathbb{R}^d$: $\mathcal{A}^{\mathrm{re}}$ consists of all sets $A$ of the form
\[
A=\{x\in\mathbb{R}^d:a_j\leq x_j\leq b_j\text{ for all }j=1,\dots,d\}
\]
for some $-\infty\leq a_j\leq b_j\leq\infty$, $j=1,\dots,d$. In particular, under suitable moment conditions, they have obtained
\begin{equation}\label{eq:cck2017}
\rho_n(\mathcal{A}^{\mathrm{re}})\leq C\lpa\frac{\log^7(dn)}{n}\rpa^{1/6},
\end{equation}
where $C>0$ is a constant independent of $n$; see Proposition 2.1 in \cite{CCK2017}. 
Indeed, they have also shown that inequality \eqref{eq:cck2017} continues to hold true with replacing $\mathcal{A}^{\mathrm{re}}$ by a class of \textit{simple convex sets} or \textit{sparsely convex sets} under appropriate assumptions; see Section 3 in \cite{CCK2017} for details.  

From \eqref{eq:cck2017}, we infer $\rho_n(\mathcal{A}^{\mathrm{re}})\to0$ as $n\to\infty$ if $(\log d)^7/n\to0$. 
Although this condition is much weaker than the ones imposed to obtain the convergence of $\rho_n(\mathcal{A}^{\mathrm{co}})$ or the Wasserstein distances, it is still unclear whether this condition is necessary to get the convergence $\rho_n(\mathcal{A}^{\mathrm{re}})\to0$ under reasonable moment conditions. In fact, in \cite{CCK2017} it is conjectured that $\log^7d$ would be replaced by $\log^3d$ in \eqref{eq:cck2017} (see Remark 2.1 in \cite{CCK2017}). 
In this paper, we show that $\log^7d$ can be replaced by $\log^5d$ in \eqref{eq:cck2017} under the \textit{same} assumptions as in \cite{CCK2017}. Moreover, if $S_n^X$ has a common factor across the components, we can further reduce $\log^7d$ in \eqref{eq:cck2017} to $\log^3d$. 
Thus, under appropriate moment conditions, we obtain $\rho_n(\mathcal{A}^{\mathrm{re}})\to0$ as $n\to\infty$ if $(\log d)^5/n\to0$ in a general setting and $(\log d)^3/n\to0$ in the presence of a common factor across the components of $S_n^X$. Note that it is still unknown whether these conditions are improvable or not in a minimax sense (see the end of Section \ref{sec:main} for a discussion). 

We shall mention that there are a few relevant studies which intend to relax the dimension growth conditions in convergences related to the above problems: \citet{DZ2020} have shown that the condition $(\log d)^5/n\to0$ is sufficient to obtain the consistency of some bootstrap approximations for $\max_{1\leq j\leq d}S^X_{n,j}$. They have also shown that the Rademacher bootstrap approximation for $\max_{1\leq j\leq d}S^X_{n,j}$ is consistent if $(\log d)^4/n\to0$ and $X_i$'s are symmetric. \citet{KMB2019} have proved $\rho_n(\mcl{A}^m)\to0$ as $n\to\infty$ under the condition $(\log d)^4/n\to0$ when the median of $\max_{1\leq j\leq d}Z^X_{n,j}$ is tight as $n\to\infty$. 
Compared to these existing results, this paper \textit{directly} improves the dimension growth conditions of some estimates obtained in \cite{CCK2017}; see Remark \ref{rmk:main} (see also Remarks \ref{rmk:anti} and \ref{rmk:DZ2020}). 

The remainder of the paper is organized as follows. 
Section \ref{sec:main} presents the main results of the paper, while Section \ref{sec:boot} develops a bootstrap approximation theorem complementing the main results in terms of statistical applications. 
Section \ref{sec:lemma} demonstrates a fundamental lemma and its proof. 
Sections \ref{sec:proof2}--\ref{sec:proof3} are devoted to the proofs for the results stated in Sections \ref{sec:main}--\ref{sec:boot}.

\section*{Notation}

Throughout the paper, we assume $d\geq3$ and $n\geq3$. 
We regard all vectors as column vectors. Given a vector $x\in\mathbb{R}^d$, we denote by $x_j$ the $j$-th coordinate of $x$, i.e.~$x=(x_1,\dots,x_d)^\top$. Here, $\top$ means transposition of a matrix. 
We write $\|x\|_{\ell_\infty}=\max_{1\leq j\leq d}|x_j|$. 
Given a sequence $X=(X_i)_{i=1}^n$ of random vectors in $\mathbb{R}^d$, we denote the $j$-th component of $X_i$ by $X_{ij}$ or $X_{i,j}$. 
For a positive integer $k$, we write $[k]:=\{1,\dots,k\}$. 
$\mcl{B}(\mathbb{R})$ denotes the Borel $\sigma$-field of $\mathbb{R}$. 
For a function $h:\mathbb{R}^d\to\mathbb{R}$, we set $\|h\|_\infty:=\sup_{x\in\mathbb{R}^d}|h(x)|$. 
$C^m_b(\mathbb{R}^d)$ denotes the space of all $C^m$ functions all of whose partial derivatives are bounded. 
We write $\partial_{j_1\dots j_r}=\frac{\partial^r}{\partial x_{j_1}\cdots\partial x_{j_r}}$ for short. 
Given a random variable $\xi$, we set $\|\xi\|_p:=\{\expe{|\xi|^p}\}^{1/p}$ for every $p>0$. Also, we define the $\psi_1$-Orlicz norm of $\xi$ by 
$
\|\xi\|_{\psi_1}:=\inf\{C>0:\expe{\psi_1(|\xi|/C)}\leq1\},
$ 
where $\psi_1(x):=\exp(x)-1$. 
For two real numbers $a$ and $b$, the notation $a\lesssim b$ means that $a\leq cb$ for some universal constant $c>0$. 

\section{Main results}\label{sec:main}

The following quantities play a key role to deduce our results:
\begin{definition}
For a random vector $F$ in $\mathbb{R}^d$, the \textit{concentration function} $\mcl{C}_F:(0,\infty)\to[0,1]$ is defined by
\[
\mcl{C}_F(\varepsilon):=\sup_{y\in\mathbb{R}^d}P\lpa0\leq\max_{1\leq j\leq d}(F_j-y_j)\leq \varepsilon\rpa,\qquad\varepsilon>0.
\]
We also set
\[
\Theta_X:=\sup_{\varepsilon>0}\varepsilon^{-1}\mcl{C}_{Z_n^X}(\varepsilon).
\]
\end{definition}
This definition of the concentration function $\mcl{C}_F$ is a multivariate extension of the one used in Section 2 of \cite[Chapter 15]{LeCam1986}. When $d=1$, $\mcl{C}_F$ is essentially the same quantity as the \textit{L\'evy concentration function} considered in \cite[Definition 1]{CCK2015}. In fact, in this case we evidently have
\[
\mcl{C}_F(2\varepsilon)=\sup_{y\in\mathbb{R}}P\lpa|F-y|\leq\ve\rpa.
\]
The quantity $\Theta_X$ measures the degree of \textit{anti-concentrations} of $Z_n^X$. 
As emphasized in \cite{CCK2013,CCK2015}, it is crucial that $Z_n^X$ exhibits reasonable anti-concentrations with respect to the dimension $d$ in order to obtain high-dimensional CLTs.  

The following is the main result of the paper. 
\begin{theorem}\label{thm:main}
Assume $\Theta_X<\infty$. 
Assume also $\max_{1\leq j\leq d}n^{-1}\sum_{i=1}^n\expe{X_{ij}^4}\leq B_n^2$ for some constant $B_n\geq1$. Then the following statements hold true:
\begin{enumerate}[label=(\alph*)]

\item\label{thm:psi} If $\max_{1\leq i\leq n}\max_{1\leq j\leq d}\|X_{ij}\|_{\psi_1}\leq B_n$, there is a universal constant $C>0$ such that 
\begin{equation*}
\rho_n(\mathcal{A}^{\mathrm{re}})
\leq C\Theta_X^{2/3}\left(\delta_{n,1}^{1/6}+\delta_{n,2}^{1/3}\right),
\end{equation*} 
where
\[
\delta_{n,1}:=\frac{B_n^2(\log d)^3}{n},\qquad
\delta_{n,2}:=\frac{B_n^2(\log d)^2(\log n)^2}{n}.
\]

\item\label{thm:mom} If $\max_{1\leq i\leq n}\|\max_{1\leq j\leq d}|X_{ij}|\|_{q}\leq D_n$ for some $q\in(2,\infty)$ and $D_n\geq1$, there is a constant $K_q>0$ which depends only on $q$ such that
\begin{align*}
\rho_n(\mathcal{A}^{\mathrm{re}})
&\leq K_q\Theta_X^{2/3}\left\{\delta_{n,1}^{1/6}+\delta_{n,2}(q)^{1/3}\right\},
\end{align*}
where
\[
\delta_{n,2}(q):=\frac{D_n^2(\log d)^{2-2/q}}{n^{1-2/q}}.
\]

\end{enumerate}
\end{theorem}
To get meaningful estimates from Theorem \ref{thm:main}, we need to bound the quantity $\Theta_X$. 
The following result, which is called \textit{Nazarov's inequality} in \cite{CCK2017}, can be used for this purpose (see \cite{CCK2017nazarov} for the proof). 
\begin{lemma}[Nazarov's inequality]\label{lemma:nazarov}
Let $Z$ be a centered Gaussian vector in $\mathbb{R}^d$ with $\ul{\sigma}:=\min_{1\leq j\leq d}\|Z_j\|_2>0$. Then, for any $\ve>0$,
\[
\mcl{C}_Z(\ve)\leq\frac{\ve}{\ul{\sigma}}(\sqrt{2\log d}+2).
\]
\end{lemma}
Theorem \ref{thm:main} and Lemma \ref{lemma:nazarov} immediately yield the following result. 
\begin{corollary}\label{coro:nazarov}
Assume $\ul{\sigma}:=\min_{1\leq j\leq d}\|S^X_{n,j}\|_2>0$. 
Then, under the assumptions of Theorem \ref{thm:main}\ref{thm:psi}, there is a universal constant $C>0$ such that 
\begin{equation*}
\rho_n(\mathcal{A}^{\mathrm{re}})
\leq \frac{C}{\ul{\sigma}^{2/3}}\left(\frac{B_n^2(\log dn)^5}{n}\right)^{1/6}.
\end{equation*} 
Also, under the assumptions of Theorem \ref{thm:main}\ref{thm:mom}, there is a constant $K_q>0$ depending only on $q$ such that
\begin{align*}
\rho_n(\mathcal{A}^{\mathrm{re}})
&\leq \frac{K_q}{\ul{\sigma}^{2/3}}\left\{\lpa\frac{B_n^2(\log d)^5}{n}\rpa^{1/6}+\lpa\frac{D_n^2(\log d)^{3-2/q}}{n^{1-2/q}}\rpa^{1/3}\right\}.
\end{align*}
\end{corollary}

\begin{rmk}\label{rmk:main}
Corollary \ref{coro:nazarov} improves the bounds given by \cite[Proposition 2.1]{CCK2017} in terms of dimension dependence under the same assumptions. In particular, we have $\rho_n(\mathcal{A}^{\mathrm{re}})\to0$ as $n\to\infty$ if $(\log d)^5/n=o(1)$, provided that $\max_{i,j}\|X_{ij}\|_{\psi_1}=O(1)$ or $\max_i\|\max_j|X_{ij}|\|_4=O(1)$. 
As a consequence, we can readily improve the dimension growth conditions in existing results obtained by applications of Proposition 2.1 in \cite{CCK2017} (or Corollary 2.1 in \cite{CCK2013}). For example, the condition $(\log p_1)^7=o(n)$ imposed in \cite[Corollary 3]{BCK2015} can be replaced by $(\log p_1)^5=o(n)$. 
Another example is Condition E in \cite[Theorem 2.1]{BCCHK2018}, where we can replace $\log^7(pn)$ by $\log^5(pn)$. 
\end{rmk}

In some situation we can bound $\Theta_X$ by a dimension-free constant. This is the case when $S_n^X$ has a common factor across the components:
\begin{lemma}\label{lemma:factor}
Let $Z$ be a centered Gaussian vector in $\mathbb{R}^d$. 
Also, let $\zeta$ be a standard Gaussian variable independent of $Z$. 
Let $a_1,\dots,a_d$ be non-zero real numbers and define $F:=(Z_1+a_1\zeta,\dots,Z_d+a_d\zeta)^\top$. Then we have
\[
\mcl{C}_F(\varepsilon)\leq\frac{2\varepsilon}{\sqrt{2\pi}\mf{a}}
\]
for any $\varepsilon>0$, where $\mf{a}:=\min_{1\leq j\leq d}|a_j|$. 
\end{lemma}
Lemma \ref{lemma:factor} is inspired by Lemma 1 in \cite[Chapter 15]{LeCam1986}. In fact, if $a_1=\cdots=a_d$, Lemma \ref{lemma:factor} is obtained as a special case of that lemma. 
\begin{corollary}\label{coro:factor}
Suppose that there is a vector $a\in\mathbb{R}^d$ such that $\mf{C}_n^X-aa^\top$ is positive semidefinite and $\mf{a}:=\min_{1\leq j\leq d}|a_j|>0$. Then, under the assumptions of Theorem \ref{thm:main}\ref{thm:psi},  there is a universal constant $C>0$ such that 
\begin{equation*}
\rho_n(\mathcal{A}^{\mathrm{re}})
\leq \frac{C}{\mf{a}^{2/3}}\left(\delta_{n,1}^{1/6}+\delta_{n,2}^{1/3}\right).
\end{equation*} 
Also, under the assumptions of Theorem \ref{thm:main}\ref{thm:mom}, there is a constant $K_q>0$ depending only on $q$ such that
\begin{align*}
\rho_n(\mathcal{A}^{\mathrm{re}})
&\leq \frac{K_q}{\mf{a}^{2/3}}\left\{\delta_{n,1}^{1/6}+\delta_{n,2}(q)^{1/3}\right\}.
\end{align*}
\end{corollary}

\begin{rmk}\label{rmk:anti}
If we restrict our attention to $\rho_n(\mcl{A}^m)$, i.e.~Gaussian approximation for $\max_{1\leq j\leq d}S_{n,j}^X$, the quantity $\Theta_X$ appearing in Theorem \ref{thm:main} can be replaced by the following one:
\[
\sup_{\ve>0}\ve^{-1}\sup_{t\in\mathbb{R}}P\lpa 0\leq \max_{1\leq j\leq d}Z_{n,j}^X-t\leq\ve\rpa.
\]
To bound this quantity, we can benefit from some recently established anti-concentration inequalities. For example, Theorem 2.2 in \cite{KMB2019} develops a dimension-free bound based on the median of $\max_{1\leq j\leq d}|Z_{n,j}^X|$, while Theorem 3.2 in \cite{BO2018} deals with a situation where $\min_{1\leq j\leq d}\|S^X_{n,j}\|_2$ is small (see also the proof of \cite[Proposition C.1]{LLM2018}). 
Such a situation can also be handled by Theorem 10 in \cite{DZ2020}, which generalizes Lemma \ref{lemma:nazarov} and can be used to get a potentially better bound for $\Theta_X$; see also Remark \ref{rmk:DZ2020}.
\end{rmk}

By Theorem \ref{thm:main}, when $\max_{i,j}\|X_{ij}\|_{\psi_1}=O(1)$ or $\max_i\|\max_j|X_{ij}|\|_4=O(1)$ as $n\to\infty$, we have $\rho_n(\mcl{A}^{\mathrm{re}})\to0$ if $\Theta_X^{4}(\log d)^3/n\to0$. 
Then, it is interesting to ask whether the condition $\Theta_X^{4}(\log d)^3/n\to0$ can be weakened or not. Thus far the answer is not known to the author's knowledge, but it might be worth mentioning that there is a situation where the condition $(\log d)^3/n\to0$ cannot be weakened: 
\begin{proposition}\label{opt-kolmogorov}
Let $\xi=(\xi_{ij})_{i,j=1}^\infty$ be an array of i.i.d.~random variables such that $\|\xi_{ij}\|_{\psi_1}<\infty$, $\expe{\xi_{ij}}=0$, $\expe{\xi_{ij}^2}=1$ and $\gamma:=\expe{\xi_{ij}^3}<0$. 
Also, let $\zeta=(\zeta_j)_{j=1}^\infty$ be a sequence of i.i.d.~standard normal variables.
Then, if the sequence $d_n\in\mathbb{N}$ satisfies $(\log d_n)^3/n\to c$ as $n\to\infty$ for some $c>0$, we have
\[
\limsup_{n\to\infty}\sup_{x\in\mathbb{R}}\labs P\lpa\max_{1\leq j\leq d_n}\frac{1}{\sqrt{n}}\sum_{i=1}^n\xi_{ij}\leq x\rpa-P\lpa\max_{1\leq j\leq d_n}\zeta_j\leq x\rpa\rabs>0.
\] 
\end{proposition}
The proof of Proposition \ref{opt-kolmogorov} is based on a Cram\'er type large deviation result, and it has already been mentioned (at least informally) in the literature; see e.g.~\cite{Hall2006} (see also Remark 1 in \cite{Chen2017}). 
Note that the proposition does not imply that the condition $\Theta_X^{4}(\log d)^3/n\to0$ is necessary because $\Theta_X$ is of order $\sqrt{\log d}$ under the assumptions of the proposition; see Example 2 in \cite{CCK2015}. 

\subsection{Comparison to \cite{KMB2019}'s results}

\cite{KMB2019} have established Gaussian approximation for $\max_{1\leq j\leq d}|S_{n,j}^X|$ under a variety of assumptions. 
In particular, their Theorem 3.2 reads as follows: Assume $X_1,\dots,X_n$ are i.i.d.~and have unit variances for simplicity. Then, for any $q\geq3$,
\begin{multline}\label{kmb-3.2}
\sup_{x\in\mathbb{R}}\left|P\left(\max_{1\leq j\leq d}|S_{n,j}^X|\leq x\right)-P\left(\max_{1\leq j\leq d}|Z_{n,j}^X|\leq x\right)\right|\\
\leq C\left(\frac{1}{2^{n}}+\mu\left(\frac{L_n^2\log^4d}{n}\right)^{1/6}+\mu\nu_q\frac{(\log d)^{1-1/q}}{n^{1/2-1/q}}\right).
\end{multline}
Here, $C$ is a universal constant and $\mu$ is the median of $\max_{1\leq j\leq d}|Z_{n,j}^X|$, while $L_n$ and $\nu_q$ are defined as follows: For every $i=1,\dots,n$, let $\zeta_i$ be the singed Borel measure on $\mathbb{R}^d$ defined by $\zeta_i(A)=P(X_i\in A)-P(Z_n^X\in A)$ for every Borel set $A$ in $\mathbb{R}^d$. Then we set
\begin{align*}
L_n&:=\frac{1}{n}\sum_{i=1}^n\max_{1\leq j\leq d}\int_{\mathbb{R}^d}|x_j|^3|\zeta_i|(dx),&
\nu_q&:=\left\{\frac{1}{n}\sum_{i=1}^n\int_{\mathbb{R}^d}\max_{1\leq j\leq d}|x_j|^q|\zeta_i|(dx)\right\}^{1/q},
\end{align*}
where $|\zeta_i|$ denotes the total variation measure of $\zeta_i$. 
To make comparison to our result possible, we consider the worst case $\mu\asymp\sqrt{\log d}$ as in Corollary \ref{coro:nazarov} in the following. 
Then, under the assumptions of Theorem \ref{thm:main}\ref{thm:mom}, $\nu_q$ can be bounded as follows:
\begin{align*}
\nu_q^q\leq\frac{1}{n}\sum_{i=1}^n\expe{\|X_i\|_{\ell_\infty}^q+\|Z_n^X\|_{\ell_\infty}^q}
\leq D_n^q+\expe{\|Z_n^X\|_{\ell_\infty}^q}.
\end{align*}
Thus, up to a constant depending only on $q$, \eqref{kmb-3.2} can be reduced to the following form:
\begin{equation}\label{kmb-3.2-r}
\left(\frac{L_n^2\log^7d}{n}\right)^{1/6}+\frac{D_n(\log d)^{3/2-1/q}}{n^{1/2-1/q}}
+\frac{\expe{\|Z_n^X\|_{\ell_\infty}^3}^{1/3}(\log d)^{3/2-1/q}}{n^{1/2-1/q}}.
\end{equation}
On one hand, this bound always exhibits a better dependence on $n$ than our result because the latter contains a bound of the form
\[
\lpa\frac{D_n(\log d)^{3/2-1/q}}{n^{1/2-1/q}}\rpa^{2/3}.
\]
On the other hand, to ensure that the bounds converge to 0 as $n\to\infty$, our result imposes conditions on $d$ and $D_n$ no worse than \eqref{kmb-3.2-r} as long as both $L_n$ and $B_n$ do not increase with $n$. In particular, when $q\geq4$ and $D_n$ does not increase with $n$ as well, our bound converges to 0 as $n\to\infty$ if $\log d=o(n^{1/5})$, requiring a weaker condition than the one to ensure convergence of \eqref{kmb-3.2-r}.  

One clear advantage of \cite{KMB2019}'s results is that they are applicable to the situation where only the third moments of $X_i$ are finite. Indeed, their results can handle the case that only the $(2+\tau)$-th moments of $X_i$ are finite for some $\tau>0$; see discussions after their Theorem 3.1. 
Meanwhile, we remark that the difference between $L_n$ and $B_n$ could be significant in some statistical applications. For example, let us consider the following nonparametric regression with regular design:
\[
y_i=f(i/n)+\epsilon_i,\qquad i=1,\dots,n,
\] 
where $f:[0,1]\to\mathbb{R}$ and $\epsilon_1,\dots,\epsilon_n$ are centered i.i.d.~variables. 
Let us estimate the function $f$ by the following kernel estimator:
\[
\hat{f}_n(x)=\frac{1}{nh_n}\sum_{i=1}^ny_iK\lpa\frac{i/n-x}{h_n}\rpa,\qquad x\in[0,1],
\]
where $K:\mathbb{R}\to\mathbb{R}$ is a kernel function satisfying $\int_{-\infty}^\infty K(u)du=1$ and $h_n$ is a bandwidth parameter which tends to 0 as $n\to\infty$. For simplicity we assume $K$ is continuous and compactly supported. 
To construct uniform confidence bands for $f$ at the points $x_1,\dots,x_d\in[0,1]$, we wish to establish Gaussian approximation for the following quantity:
\[
\max_{1\leq j\leq d}\sqrt{nh_n}\left|\hat{f}_n(x_j)-\expe{\hat{f}_n(x_j)}\right|.
\]
For this purpose we can apply a high-dimensional CLT for $S_n^X$ with $X_{ij}=h_n^{-1/2}\epsilon_iK((i/n-x_j)/h_n)$. In this case, $L_n$ is typically of order $h_n^{-3/2}$, while
\[
\frac{1}{n}\sum_{i=1}^n\expe{X_{ij}^4}\approx\expe{\epsilon_1^4}\cdot h_n^{-2}\int_0^1 K\lpa\frac{y-x_j}{h_n}\rpa^4dy
\leq \expe{\epsilon_1^4}\cdot h_n^{-1}\int K(u)^4du,
\] 
so $B_n$ is of order $h_n^{-1/2}$ as long as $\expe{\epsilon_1^4}<\infty$. Note that $D_n$ is also of order $h_n^{-1/2}$ in this situation. 
See also Remark 4.8 of \cite{Koike2017stein} for another application of this type. 


We should also mention that \cite{KMB2019} also deal with the case that $X_{ij}$ are sub-Weibull for all $i,j$, which generalizes the assumptions of Theorem \ref{thm:main}\ref{thm:psi}; see their Corollary 3.1 for details. 
Moreover, they have also developed high-dimensional versions of non-uniform CLTs and Cram\'er type large deviations. These topics are beyond the scope of this paper.

\section{Bootstrap approximation}\label{sec:boot}

In terms of statistical applications, the Gaussian approximation results obtained in the previous section are infeasible unless the covariance matrix $\mf{C}_n^X$ is known for statisticians. Moreover, even if this is the case, the probability $P(Z_n^X\in A)$ is analytically intractable for a general set $A\in\mcl{A}^{\mathrm{re}}$. For these reasons, this section develops bootstrap approximation for $P(Z_n^X\in A)$ with $A\in\mcl{A}^{\mathrm{re}}$, following \cite{CCK2017}. 


Let $w=(w_i)_{i=1}^n$ be a sequence of independent random variables independent of $X$. 
We consider the wild bootstrap (also called the multiplier bootstrap) with multiplier variables $w$ as
\[
S_n^{\mathrm{WB}}=(S_{n,1}^{\mathrm{WB}},\dots,S_{n,d}^{\mathrm{WB}})^\top:=\frac{1}{\sqrt{n}}\sum_{i=1}^nw_i\lpa X_i-\bar{X}\rpa,
\]
where $\bar{X}:=n^{-1}\sum_{i=1}^nX_i$. We set
\[
\rho_n^{\mathrm{WB}}(\mcl{A}^{\mathrm{re}}):=\sup_{A\in\mathcal{A}^{\mathrm{re}}}\labs P\lpa S_n^{\mathrm{WB}}\in A\mid X\rpa-P\lpa Z_n^X\in A\rpa\rabs.
\]
\begin{theorem}\label{thm:wild}
Suppose that $\expe{w_i}=0$ and $\expe{w_i^2}=1$ for every $i$. Suppose also that there is a constant $b\geq1$ such that $|w_i|\leq b$ a.s.~for every $i$. Then the following statements hold true:
\begin{enumerate}[label=(\alph*)]

\item\label{thm:wild-psi} Under the assumptions of Theorem \ref{thm:main}\ref{thm:psi}, there is a universal constant $C'>0$ such that
\begin{equation*}
\expe{\rho_n^{\mathrm{WB}}(\mcl{A}^{\mathrm{re}})}
\leq C'\Theta_X^{2/3}\left((b^2\delta_{n,1})^{1/6}+(b^2\delta_{n,2})^{1/3}\right).
\end{equation*}

\item\label{thm:wild-mom} Under the assumptions of Theorem \ref{thm:main}\ref{thm:mom}, there is a constant $K_q'>0$ depending only on $q$ such that
\begin{equation*}
\expe{\rho_n^{\mathrm{WB}}(\mcl{A}^{\mathrm{re}})}
\leq K_q'\Theta_X^{2/3}\left((b^2\delta_{n,1})^{1/6}+(b^2\delta_{n,2}(q))^{1/3}\right).
\end{equation*}

\end{enumerate}
\end{theorem}

Theorem \ref{thm:wild} and Lemma \ref{lemma:nazarov} yield the following counterpart of Corollary \ref{coro:nazarov}. 
\begin{corollary}\label{coro:wild}
Assume $\ul{\sigma}:=\min_{1\leq j\leq d}\|S^X_{n,j}\|_2>0$. 
Then, under the assumptions of Theorem \ref{thm:wild}\ref{thm:wild-psi},  there is a universal constant $C'>0$ such that 
\begin{equation}\label{eq:wild-psi}
\expe{\rho_n^{\mathrm{WB}}(\mcl{A}^{\mathrm{re}})}
\leq \frac{C'}{\ul{\sigma}^{2/3}}\left(\frac{b^2B_n^2(\log dn)^5}{n}\right)^{1/6}.
\end{equation} 
Also, under the assumptions of Theorem \ref{thm:wild}\ref{thm:wild-mom}, there is a constant $K'_q>0$ depending only on $q$ such that
\begin{equation}\label{eq:wild-mom}
\expe{\rho_n^{\mathrm{WB}}(\mcl{A}^{\mathrm{re}})}
\leq \frac{K'_q}{\ul{\sigma}^{2/3}}\left\{\lpa\frac{b^2B_n^2(\log d)^5}{n}\rpa^{1/6}+\lpa\frac{b^2D_n^2(\log d)^{3-2/q}}{n^{1-2/q}}\rpa^{1/3}\right\}.
\end{equation}
\end{corollary}
It is of course possible to derive a bootstrap counterpart of Corollary \ref{coro:factor}. We omit the precise statement. 


\begin{rmk}[Relation to \cite{CCK2017}]\label{rmk:cck-boot}
\citet{CCK2017} have established similar results to Corollary \ref{coro:wild} when $w$ is Gaussian. Indeed, they have derived stronger results that the probabilities of $\rho_n^{\mathrm{WB}}(\mcl{A}^{\mathrm{re}}(d))$ exceeding the right hand sides of \eqref{eq:wild-psi} or \eqref{eq:wild-mom} are small; see Proposition 4.1 in \cite{CCK2017} for details. It is presumably possible to obtain similar results in our case by showing that the variables appearing in the right side of \eqref{coupling-boot}  concentrate at their expectations. 
\end{rmk}

\begin{rmk}[Relation to \cite{DZ2020}]\label{rmk:DZ2020}
\citet{DZ2020} have developed analogous results to Corollary \ref{coro:wild} for 
\[
\sup_{A\in\mathcal{A}^m}\labs P\lpa S_n^{\mathrm{WB}}\in A\mid X\rpa-P\lpa S_n^X\in A\rpa\rabs
\]
instead of $\rho_n^{\mathrm{WB}}(\mcl{A}^{\mathrm{re}})$ and assuming $w_i$ are i.i.d.~and sub-Gaussian (rather than bounded) but satisfies $\expe{w_i^3}=1$. 
In particular, under the assumptions of Theorem \ref{thm:main}\ref{thm:psi}, we can derive the following bound from their Corollary 1(i):
\begin{equation}\label{dz-1}
\ex{\sup_{A\in\mathcal{A}^m}\labs P\lpa S_n^{\mathrm{WB}}\in A\mid X\rpa-P\lpa S_n^X\in A\rpa\rabs}
\leq \frac{C^*}{\ol{\sigma}^{2/3}}\left(\frac{B_n^2(\log dn)^5}{n}\right)^{1/6},
\end{equation}
where $C^*$ depends only on the sub-Gaussian parameter of $w_1$ and
\[
\ol{\sigma}:=\min_{1\leq j\leq d}\frac{2+\sqrt{2\log d}}{1/\sigma_{(1)}+(1+\sqrt{2\log j})/\sigma_{(j)}}
\]
with $\sigma_{(1)}\leq\cdots\leq\sigma_{(d)}$ being the ordered diagonal entries of $\mf{C}^X_n$. 
We can also derive the following bound from their Corollary 3(ii) under the assumptions of Theorem \ref{thm:main}\ref{thm:mom} with taking $\delta=\lpa\frac{B_n^2(\log dn)^5}{\ol{\sigma}^4n}\rpa^{1/6}$ in their result (note that $B_n^2/\ol{\sigma}^4\geq1$ by Jensen's inequality):
\begin{equation}\label{dz-2}
\ex{\sup_{A\in\mathcal{A}^m}\labs P\lpa S_n^{\mathrm{WB}}\in A\mid X\rpa-P\lpa S_n^X\in A\rpa\rabs}
\leq \frac{K_q^*}{\ol{\sigma}^{2/3}}\lpa\frac{B_n^2(\log dn)^5}{n}\rpa^{1/6},
\end{equation}
where $K_q^*$ depends only on $q$ and the sub-Gaussian parameter of $w_1$.  
Since Theorem 10 in \cite{DZ2020} implies $\Theta_X\lesssim\ol{\sigma}^{-1}\sqrt{\log d}$, our Theorems \ref{thm:main}\ref{thm:psi} and \ref{thm:wild}\ref{thm:wild-psi} lead to a bound similar to \eqref{dz-1}. 
In the meantime, the bound obtained from our Theorems \ref{thm:main}\ref{thm:mom} and \ref{thm:wild}\ref{thm:wild-mom} is worse than \eqref{dz-2} due to the presence of the second term.  
The advantage of our results over \cite{DZ2020}'s ones is that we do not need to assume $\expe{w_i^3}=1$. 
We note that they have indeed established stronger estimates as the one stated in Remark \ref{rmk:cck-boot} and allow $w_i$ to be sub-Gaussian (rather than bounded). In this paper, we do not pursue such generalization for simplicity.  
\end{rmk}

\begin{rmk}[Empirical bootstrap]
It seems difficult to derive a result comparable to Theorem \ref{thm:wild} for Efron's empirical bootstrap using our proof technique. This is because we need to bound a quantity of the form $\expe{\sum_{i=1}^n\max_{1\leq j\leq d}X_{ij}^4}$ rather than $\expe{\max_{1\leq j\leq d}\sum_{i=1}^nX_{ij}^4}$ while we apply our key Lemma \ref{coupling} in order to derive such a result for Efron's empirical bootstrap. 
We shall remark that this issue has also been pointed out in \cite{DZ2020} (see discussions after Theorem 2 of the paper). 
\end{rmk}

\section{Fundamental lemma}\label{sec:lemma}

The basic strategy for the proofs of the main results is the same as the one used in the proof of \cite[Theorem 2.2]{CCK2013}, which is based on \cite[Theorem 2.1]{CCK2013} and an anti-concentration inequality. 
Here, since we do not explicitly bound the quantity $\Theta_X$, we need to establish only a counterpart of the former. 
This part is the main technical development of this paper and the result is given as follows:
%
\begin{lemma}\label{coupling}
Let $Z$ be a centered Gaussian vector in $\mathbb{R}^d$ with covariance matrix $\mf{C}=(\mf{C}_{jk})_{1\leq j,k\leq d}$. 
Suppose that $\expe{X_{ij}^4}<\infty$ for all $i,j$. 
Then, there is a universal constant $C>0$ such that
\begin{multline}\label{eq:coupling}
P\lpa\max_{1\leq j\leq d}(S_{n,j}^X-y_j)\in A\rpa
\leq P\lpa\max_{1\leq j\leq d}(Z_j-y_j)\in A^{5\varepsilon}\rpa\\
+C\left\{\varepsilon^{-2}\lpa\Delta^X_{n,0}\log d+\Delta^X_{n,1}\sqrt{\frac{(\log d)^3}{n}}\rpa
+\varepsilon^{-4}\Delta^X_{n,2}(\varepsilon)\frac{(\log d)^3}{n}\right\}
\end{multline}
for any $y\in\mathbb{R}^d$, $\varepsilon>0$ and $A\in\mcl{B}(\mathbb{R})$, where
\begin{align*}
\Delta^X_{n,0}&:=\max_{1\leq j,k\leq d}\labs\frac{1}{n}\sum_{i=1}^n\expe{X_{ij}X_{ik}}-\mf{C}_{jk}\rabs,&
\Delta^X_{n,1}&:=\sqrt{\frac{1}{n}\ex{\max_{1\leq j\leq d}\sum_{i=1}^nX_{ij}^4}}
\end{align*}
and
\begin{align*}
\Delta^X_{n,2}(\varepsilon)&:=\frac{1}{n}\sum_{i=1}^n\ex{\|X_{i}\|_{\ell_\infty}^4;\|X_{i}\|_{\ell_\infty}>\sqrt{n}\varepsilon/(3\log d)}.
\end{align*}
\end{lemma}
Lemma \ref{coupling} can be seen as a variant of \cite[Theorem 4.1]{CCK2014} and \cite[Theorem 3.1]{CCK2016}, and it is closely related to Gaussian couplings for $\max_{1\leq j\leq d}(S_{n,j}^X-y_j)$; see Lemma 4.1 in \cite{CCK2014}. 
The proof strategy is basically the same as these two theorems and consists of the following two steps: First, we approximate the indicator function $1_A$ and the maximum function by appropriate smooth functions. Second, we estimate $|\expe{g(S_n^X-y)}-\expe{g(Z-y)}$ for a particular class of smooth functions $g$ and establish their ``good'' bounds with respect to $d$. 
To get good bounds in the second step, we partially follow the idea of \cite{DZ2020}, where a randomized version of the Lindeberg method is developed to improve the dimension dependence of bootstrap approximations for $\max_{1\leq j\leq d}S^X_{n,j}$. 
To transfer this improvement to Gaussian approximations for $\max_{1\leq j\leq d}S^X_{n,j}$, we show that the dimension dependence of Gaussian approximations is improvable for a specific wild bootstrap version of $S_n^X$ by the Stein kernel method. 
All together, we will complete the proof of Lemma \ref{coupling}. 

Indeed, the main contribution of this paper is the final part in the above. To be precise, we can prove the following result:  
\begin{proposition}\label{prop:beta}
Let $(\eta_i)_{i=1}^n$ be a sequence of i.i.d.~random variables independent of $X$ such that the law of $\eta_i$ is the beta distribution with parameters $1/2,3/2$ for every $i$. That is, $\eta_i$ has the density function of the form $f(x)=\mathrm{B}(1/2,3/2)^{-1}\sqrt{(1-x)/x}1_{(0,1)}(x)$, where $\mathrm{B}(u,v)$ denotes the beta function. 
Set $\xi_i:=4\eta_i-1$ and $Y_i:=\xi_iX_i$ for every $i=1,\dots,n$ and let $Y=(Y_i)_{i=1}^n$. Then, under the assumptions of Lemma \ref{coupling}, 
\[
\rho_{h,\beta}\lpa S_n^Y,Z\rpa\leq\frac{3}{2}\lpa\max_{1\leq l\leq 2}\beta^{2-l}\|h^{(l)}\|_\infty\rpa\left\{
\Delta^X_{n,0}
+5\sqrt{2}\sqrt{\frac{\log(2d^2)}{n}}\Delta^X_{n,1}
\right\}
\]
for any $h\in C_b^\infty(\mathbb{R})$ and $\beta>0$, where $\rho_{h,\beta}\lpa S_n^Y,Z\rpa$ is defined by \eqref{def:rho}. 
\end{proposition}
Proposition \ref{prop:beta} is proved at the end of Section \ref{sec:stein}. 
Roughly speaking, this result means that a special wild bootstrap version of $S_n^X$ may enjoy Gaussian approximation with an improved dimension dependence condition. 
Here, the key fact is that the multiplier sequence $\xi_i$ in the above satisfies $\expe{\xi_i}=0$ and $\expe{\xi_i^2}=\expe{\xi_i^3}=1$. Hence, we can consider a Lindeberg type interpolation between $S_n^Y$ and $S_n^X$ up to the \textit{third} moments rather than the \textit{second} moments. 
This is the main difference between our method and the Stein-Slepian approach by \cite{CCK2017} because the latter allows moment matching only up to the second moments. 

In fact, if $\|\max_{1\leq j\leq d}|X_{ij}|\|_4$ is bounded by a constant $B_n$ uniformly in $i$, we can use a standard Lindeberg method to obtain the following result, which is weaker than Lemma \ref{coupling} but still leads to the dimension improvement in some cases. 
\begin{proposition}\label{s-coupling}
Suppose that there is a constant $B_n$ satisfying $\|\max_{1\leq j\leq d}|X_{ij}|\|_4\leq B_n$ for all $i=1,\dots,n$. Under the assumptions of Lemma \ref{coupling},
\[
P\lpa\max_{1\leq j\leq d}(S_{n,j}^X-y_j)\in A\rpa
\leq P\lpa\max_{1\leq j\leq d}(Z_j-y_j)\in A^{5\varepsilon}\rpa
+C\varepsilon^{-2}\lpa\Delta^X_{n,0}\log d+\sqrt{\frac{B_n^4(\log d)^3}{n}}\rpa
\]
for any $y\in\mathbb{R}^d$, $\varepsilon>0$ and $A\in\mcl{B}(\mathbb{R})$. 
\end{proposition}
We can indeed use this result to improve the dimension dependence of Gaussian approximation for $S_n^X$: 
\begin{corollary}\label{weak-GA}
Assume $\Theta_X<\infty$. Under the assumptions of Proposition \ref{s-coupling}, we have
\[
\rho_n(\mcl{A}^{\mathrm{re}})\leq C\Theta_X^{2/3}\lpa\frac{B_n^4(\log d)^3}{n}\rpa^{1/6},
\]
where $C>0$ is a universal constant. 
In particular, if $\ul{\sigma}:=\min_{1\leq j\leq d}\|S^X_{n,j}\|_2>0$, there is a universal constant $C'>0$ such that 
\begin{equation*}
\rho_n(\mathcal{A}^{\mathrm{re}})
\leq \frac{C'}{\ul{\sigma}^{2/3}}\left(\frac{B_n^4(\log d)^5}{n}\right)^{1/6}.
\end{equation*} 
\end{corollary}
In Appendix \ref{sec:appendix}, we give proofs for Proposition \ref{s-coupling} and Corollary \ref{weak-GA} by the standard Lindeberg method. 
We note that, if we additionally assume $\max_{1\leq j\leq d}n^{-1}\sum_{i=1}^n\expe{X_{ij}^4}\leq B_n^2$ in this situation, we can replace $B_n^4$ in the above bound by $B_n^2$ using the randomized Lindeberg method. This is a typical advantage of using the randomized Lindeberg method in our proofs. 
An intuitive explanation is as follows. First, the standard Lindeberg argument leads to a bound of the form
\begin{equation*}
\frac{1}{n^2}\sum_{i=1}^n\ex{\max_{1\leq j\leq d}X_{ij}^4},
\end{equation*}
which is bounded by $B_n^4/n$; see Lemma \ref{lemma:s-lindeberg}. 
Indeed, utilizing the stability property of the smooth max function given by Lemma \ref{cck-derivative}(iii), we can reduce the above bound to the following one (plus some reminder terms):
\begin{equation}\label{naive}
\frac{1}{n^2}\sum_{i=1}^n\max_{1\leq j\leq d}\ex{X_{ij}^4}.
\end{equation}
This is a standard argument in the Chernozhukov-Chetverikov-Kato theory. 
Now, in this expression, the sum ``$\sum_{i=1}^n$'' comes from the Lindeberg interpolation with replacing $X_1,\dots,X_n$ by $Y_1,\dots,Y_n$ (defined in Proposition \ref{prop:beta}) one by one \textit{with this order}. 
However, there is no reason to keep the order of replacement because $\sum_{i=1}^nX_i$ is invariant under permutation of $X_1,\dots,X_n$. 
This intuitively suggests that we could change the bound \eqref{naive} to $n^{-1}\max_{1\leq j\leq d}\expe{X_{I,j}^4}$ by randomizing the order of replacement, where $I$ is a uniform variable over $\{1,\dots,n\}$ independent of $X$. This quantity is bounded by $B_n^2/n$ because
\[
\frac{1}{n}\max_{1\leq j\leq d}\expe{X_{I,j}^4}
=\frac{1}{n^2}\max_{1\leq j\leq d}\sum_{i=1}^n\expe{X_{ij}^4}.
\]
\begin{rmk}[Application to empirical processes]
As developed in \cite{CCK2014,CCK2016}, it will be possible to apply Lemma \ref{coupling} for obtaining Gaussian approximations for suprema of empirical processes. 
We remark that this could improve the convergence rate of such an approximation since the term multiplied by $\ve^{-4}$ is often dominated by the term multiplied by $\ve^{-2}$ in \eqref{eq:coupling} under suitable moment conditions as in Lemmas \ref{lemma:psi}--\ref{lemma:mom}; see Remark 4.8 in \cite{Koike2017stein} for an explanation of why this improves the convergence rate. 
Nevertheless, this topic is beyond the scope of this paper and left for future work. 
\end{rmk}
The remainder of this section is devoted to the proof of Lemma \ref{coupling}. 

\subsection{Smooth approximation}

We begin by approximating the indicator function $1_A$ and the maximum function by smooth functions. For the indicator function, we will use the following result. 
\begin{lemma}[\cite{CCK2016}, Lemma 5.1]\label{cck-approx}
For any $\varepsilon>0$ and Borel set $A$ of $\mathbb{R}$, there is a $C^\infty$ function $h:\mathbb{R}\to\mathbb{R}$ satisfying the following conditions:
\begin{enumerate}[label=(\roman*)]

\item\label{deriv-bound} There is a universal constant $C>0$ such that $\|h^{(r)}\|_\infty\leq C\varepsilon^{-r}$ for $r=1,2,3,4$.

\item $1_A(x)\leq h(x)\leq 1_{A^{3\varepsilon}}(x)$ for all $x\in\mathbb{R}$.

\end{enumerate}
\end{lemma}

\begin{rmk}
Formally, Lemma 5.1 in \cite{CCK2016} states that condition \ref{deriv-bound} in the above is satisfied only for $r=1,2,3$, but the function constructed there indeed satisfies this condition for $r=4$. 
\end{rmk}

Next we introduce the following special form of smooth approximation of the maximum function: For each $\beta>0$, we define the function $\Phi_\beta:\mathbb{R}^d\to\mathbb{R}$ by 
\[
\Phi_\beta(x)=\beta^{-1}\log\left(\sum_{j=1}^de^{\beta x_j}\right).
\]
This ``smooth max function'' is one of the key constituents in the Chernozhukov-Chetverikov-Kato theory. 
One can easily verify the following inequality (cf.~Eq.(1) in \cite{CCK2015}):
\begin{equation}\label{max-smooth}
0\leq \Phi_\beta(x)-\max_{1\leq j\leq d}x_j\leq \beta^{-1}\log d
\end{equation}
for any $x\in\mathbb{R}^d$. Thus, $\Phi_\beta$ better approximates the maximum function as the value of $\beta$ increases. 
The next lemma summarizes Lemmas 5--6 in \cite{DZ2020} and highlights the key properties of this smooth max function: 
\begin{lemma}\label{cck-derivative}
For any $\beta>0$, $m\in\mathbb{N}$ and $C^m$ function $h:\mathbb{R}\to\mathbb{R}$, there is an $\mathbb{R}^{\otimes m}$-valued function $\Upsilon_{\beta}(x)=(\Upsilon^{j_1,\dots,j_m}_\beta(x))_{1\leq j_1,\dots,j_m\leq d}$ on $\mathbb{R}^d$ satisfying the following conditions:
\begin{enumerate}[label=(\roman*)]

\item For any $x\in\mathbb{R}^d$ and $j_1,\dots,j_m\in[d]$, we have 
$
|\partial_{j_1\dots j_m}(h\circ\Phi_\beta)(x)|\leq \Upsilon_\beta^{j_1,\dots,j_m}(x).
$

\item For every $x\in\mathbb{R}^d$, we have
\begin{align*}
\sum_{j_1,\dots,j_m=1}^d\Upsilon_\beta^{j_1,\dots, j_m}(x)
\leq c_{m}\max_{1\leq k\leq m}\beta^{m-k}\|h^{(k)}\|_\infty,
\end{align*}
where $c_{m}>0$ depends only on $m$. 

\item For any $x,t\in\mathbb{R}^d$ and $j_1,\dots,j_m\in[d]$, we have
\[
e^{-8\|t\|_{\ell_\infty}\beta}\Upsilon_\beta^{j_1,\dots,j_m}(x+t)\leq \Upsilon_\beta^{j_1,\dots,j_m}(x)\leq e^{8\|t\|_{\ell_\infty}\beta}\Upsilon_\beta^{j_1,\dots,j_m}(x+t).
\]

\end{enumerate}
\end{lemma}

As a result, given two random variables $F$ and $G$, we explore bounds for the quantity
\begin{equation}\label{def:rho}
\rho_{h,\beta}(F,G):=\sup_{y\in\mathbb{R}^d}\left|\ex{h\left(\Phi_\beta(F-y)\right)}-\ex{h\left(\Phi_\beta(G-y)\right)}\right|
\end{equation}
for a (smooth) bounded function $h:\mathbb{R}\to\mathbb{R}$ and $\beta>0$ in the following. 

\subsection{Randomized Lindeberg method}

The next lemma essentially has the same content as \cite[Theorem 5]{DZ2020}. For the sake of completeness, we give a self-contained proof. 

In the following, we will use the standard multi-index notation: For a multi-index $\lambda=(\lambda_1,\dots,\lambda_d)\in\mathbb{Z}_+^d$, we set $|\lambda|:=\lambda_1+\cdots+\lambda_d$, $\lambda!:=\lambda_1!\cdots\lambda_d!$ and $\partial^\lambda:=\partial_1^{\lambda_1}\cdots\partial_d^{\lambda_d}$ as usual. Also, given a vector $x=(x_1,\dots,x_d)\in\mathbb{R}^d$, we write $x^\lambda=x_1^{\lambda_1}\cdots x_d^{\lambda_d}$.  
\begin{lemma}\label{lemma:lindeberg}
Let $X=(X_i)_{i=1}^n$ and $Y=(Y_i)_{i=1}^n$ be two sequences of independent centered random vectors in $\mathbb{R}^d$. 
Suppose that there is an integer $m\geq3$ such that $\expe{|X_{ij}|^m+|Y_{ij}|^m}<\infty$ for all $i\in[N]$, $j\in[d]$ and $\expe{X_i^\lambda}=\expe{Y_i^\lambda}$ for all $i\in[N]$ and $\lambda\in\mathbb{Z}_+^d$ with $|\lambda|\leq m-1$. 
Then, for any $h\in C^m_b(\mathbb{R})$ and $\beta>0$, we have
\begin{align*}
&\rho_{h,\beta}(S_n^X,S_n^Y)\nonumber\\
&\leq C_mn^{-\frac{m}{2}}\lpa\max_{1\leq l\leq m}\beta^{m-l}\|h^{(l)}\|_\infty\rpa\left\{\max_{1\leq j\leq d}\sum_{i=1}^n\expe{|X_{ij}|^m+|Y_{ij}|^m}
+\sum_{i=1}^n\expe{\|X_{i}\|_{\ell_\infty}^m;\|X_{i}\|_{\ell_\infty}>\sqrt{n}/\beta}
\right.\\
&\left.\quad+\sum_{i=1}^n\expe{\|Y_{i}\|_{\ell_\infty}^m;\|Y_{i}\|_{\ell_\infty}>\sqrt{n}/\beta}
+\sum_{i=1}^nP\lpa\|X_{i}\|_{\ell_\infty}\vee\|Y_{i}\|_{\ell_\infty}>\sqrt{n}/\beta\rpa \max_{1\leq j\leq d}\expe{|X_{ij}|^m+|Y_{ij}|^m}\right\},
\end{align*}
where $C_m>0$ depends only on $m$. 
\end{lemma}

\begin{proof}
We may assume that $X$ and $Y$ are independent without loss of generality. 
Throughout the proof, for two real numbers $a$ and $b$, the notation $a\lesssim_m b$ means that $a\leq c_mb$ for some constant $c_m>0$ which depends only on $m$. 

Take a vector $y\in\mathbb{R}^d$ and define the function $\Psi:\mathbb{R}^d\to\mathbb{R}$ by $\Psi(x)=h(\Phi_\beta(x-y))$ for $x\in\mathbb{R}^d$. 
We denote by $\mf{S}_n$ the set of all permutations of $[n]$. 
For any $\sigma\in\mf{S}_n$ and $k\in[n]$, we set 
\[
S_n^\sigma(k):=\frac{1}{\sqrt{n}}\sum_{i=1}^kX_{\sigma(i)}+\frac{1}{\sqrt{n}}\sum_{i=k+1}^nY_{\sigma(i)}\quad
\text{and}
\quad
\wh{S}_n^\sigma(k):=S_n^\sigma(k)-\frac{1}{\sqrt{n}}X_{\sigma(k)}.
\]
By construction $\wh{S}_n^\sigma(k)$ is independent of $X_{\sigma(k)}$ and $Y_{\sigma(k)}$. 
We also have $S_n^\sigma(k)=\wh{S}_n^\sigma(k)+n^{-1/2}X_{\sigma(k)}$ and $S_n^\sigma(k-1)=\wh{S}_n^\sigma(k)+n^{-1/2}Y_{\sigma(k)}$ (with $S_n^\sigma(0):=n^{-1/2}\sum_{i=1}^nY_{\sigma(i)}$). 
Moreover, it holds that $S_n^\sigma(n)=S_n^X$ and $S_n^\sigma(0)=S_n^Y$. 
Therefore, we have
\begin{align}
\left|\ex{\Psi(S_{n}^X)}-\ex{\Psi(S_n^Y)}\right|
&=\frac{1}{n!}\sum_{\sigma\in\mf{S}_n}\left|\ex{\Psi(S_n^\sigma(n))}-\ex{\Psi(S_n^\sigma(0))}\right|\nonumber\\
&\leq\frac{1}{n!}\sum_{\sigma\in\mf{S}_n}\sum_{k=1}^n\left|\ex{\Psi(S_n^\sigma(k))}-\ex{\Psi(S_n^\sigma(k-1))}\right|.\label{lindeberg-eq1}
\end{align}

Now, when $W=X$ or $W=Y$, Taylor's theorem and the independence of $W_{\sigma(k)}$ from $\wh{S}_n^\sigma(k)$ yield
\[
\ex{\Psi\lpa\wh{S}_n^\sigma(k)+n^{-1/2}W_{\sigma(k)}\rpa}
=\sum_{\lambda\in\mathbb{Z}_+^d:|\lambda|\leq m-1}\frac{n^{-|\lambda|/2}}{\lambda!}\ex{\partial^\lambda\Psi\lpa\wh{S}_n^\sigma(k)\rpa}\ex{W_{\sigma(k)}^{\lambda}}
+R_k^\sigma[W],
\]
where
\[
R_k^\sigma[W]:=n^{-m/2}\sum_{\lambda\in\mathbb{Z}_+^d:|\lambda|=m}\frac{m}{\lambda!}\int_0^1(1-t)^{m-1}\ex{\partial^\lambda\Psi\lpa\wh{S}_n^\sigma(k)+tn^{-1/2}W_{\sigma(k)}\rpa W_{\sigma(k)}^\lambda}dt.
\]
Since we have $\expe{X_{i}^\lambda}=\expe{Y_{i}^\lambda}$ for all $i\in[N]$ and $\lambda\in\mathbb{Z}_+^d$ with $|\lambda|\leq m-1$ by assumption, we obtain
\begin{align}
\labs\ex{\Psi\lpa S^\sigma_n(k)\rpa}-\ex{\Psi\lpa S^\sigma(k-1)\rpa}\rabs
\leq|R_k^\sigma[X]|+|R_k^\sigma[Y]|.\label{lindeberg-eq2}
\end{align}
We estimate $R_k^\sigma[W]$ as
\begin{align}
|R_k^\sigma[W]|
\leq\mathbf{I}^\sigma_k[W]+\mathbf{II}^\sigma_k[W],\label{est-R}
\end{align}
where
\begin{align*}
\mathbf{I}^\sigma_k[W]&:= n^{-\frac{m}{2}}\sum_{\lambda\in\mathbb{Z}_+^d:|\lambda|=m}\frac{m}{\lambda!}\int_0^1(1-t)^{m-1}\ex{\labs\partial^\lambda\Psi\lpa\wh{S}_n^\sigma(k)+tn^{-1/2}W_{\sigma(k)}\rpa W_{\sigma(k)}^\lambda\rabs;\|W_{\sigma(k)}\|_{\ell_\infty}\leq\sqrt{n}/\beta}dt,\\
\mathbf{II}^\sigma_k[W]&:=n^{-\frac{m}{2}}\sum_{\lambda\in\mathbb{Z}_+^d:|\lambda|=m}\frac{m}{\lambda!}\int_0^1(1-t)^{m-1}\ex{\labs\partial^\lambda\Psi\lpa\wh{S}_n^\sigma(k)+tn^{-1/2}W_{\sigma(k)}\rpa W_{\sigma(k)}^\lambda\rabs;\|W_{\sigma(k)}\|_{\ell_\infty}>\sqrt{n}/\beta}dt.
\end{align*}
Lemma \ref{cck-derivative} yields
%
\begin{equation}\label{est-II}
\frac{1}{n!}\sum_{\sigma\in\mathfrak{S}_n}\sum_{k=1}^n\mathbf{II}_k^\sigma[W]
\lesssim_m n^{-\frac{m}{2}}\max_{1\leq l\leq m}\beta^{m-l}\|h^{(l)}\|_\infty\sum_{i=1}^n\expe{\|W_{i}\|_{\ell_\infty}^m;\|W_{i}\|_{\ell_\infty}>\sqrt{n}/\beta}.
\end{equation}
Now we focus on $\mathbf{I}^\sigma_k[W]$. 
We rewrite it as
\begin{multline*}
\mathbf{I}^\sigma_k[W]=\frac{n^{-\frac{m}{2}}}{(m-1)!}\sum_{j_1,\dots,j_m=1}^d\int_0^1(1-t)^{m-1}\\
\times\ex{\labs\partial_{j_1\dots j_m}\Psi\lpa\wh{S}_n^\sigma(k)+tn^{-1/2}W_{\sigma(k)}\rpa \prod_{l=1}^mW_{\sigma(k),j_l}\rabs;\|W_{\sigma(k)}\|_{\ell_\infty}\leq\sqrt{n}/\beta}dt.
\end{multline*}
We have by the AM-GM inequality 
\[
\labs\prod_{l=1}^mW_{\sigma(k),j_l}\rabs\leq\frac{1}{m}\sum_{l=1}^m\labs W_{\sigma(k),j_l}\rabs^m.
\]
Inserting this inequality into the first formula and noting that $\partial_{j_1\dots j_m}\Psi$ does not depend on the order of $j_1,\dots,j_m$, we obtain
\begin{multline*}
\mathbf{I}^\sigma_k[W]\leq\frac{n^{-\frac{m}{2}}}{(m-1)!}\sum_{j_1,\dots,j_m=1}^d\int_0^1(1-t)^{m-1}\\
\times\ex{\labs\partial_{j_1\dots j_m}\Psi\lpa\wh{S}_n^\sigma(k)+tn^{-1/2}W_{\sigma(k)}\rpa\rabs \labs W_{\sigma(k),j_1}\rabs^m;\|W_{\sigma(k)}\|_{\ell_\infty}\leq\sqrt{n}/\beta}dt.
\end{multline*}
Thus, by Lemma \ref{cck-derivative} we obtain
\begin{align}
\mathbf{I}^\sigma_k[W]
&\lesssim n^{-\frac{m}{2}}\sum_{j_1,\dots,j_m=1}^d\ex{\labs\Upsilon^{j_1,\dots,j_m}\lpa\wh{S}_n^\sigma(k)-y\rpa\rabs |W_{\sigma(k),j_1}|^m}\nonumber\\
&=n^{-\frac{m}{2}}\sum_{j_1,\dots,j_m=1}^d\ex{\labs\Upsilon^{j_1,\dots,j_m}\lpa\wh{S}_n^\sigma(k)-y\rpa\rabs} \expe{|W_{\sigma(k),j_1}|^m}\nonumber\\
&=n^{-\frac{m}{2}}\sum_{j_1,\dots,j_m=1}^d\left\{\ex{\labs\Upsilon^{j_1,\dots,j_m}\lpa\wh{S}_n^\sigma(k)-y\rpa\rabs;\mcl{E}_{\sigma,k}}
+\ex{\labs\Upsilon^{j_1,\dots,j_m}\lpa\wh{S}_n^\sigma(k)-y\rpa\rabs;\mcl{E}_{\sigma,k}^c} \right\}\expe{|W_{\sigma(k),j_1}|^m}\nonumber\\
&=:\mathbf{A}_k^\sigma[W]+\mathbf{B}_k^\sigma[W],\label{est-I-pre}
\end{align}
where $\mcl{E}_{\sigma,k}:=\{\|X_{\sigma(k)}\|_{\ell_\infty}\vee\|Y_{\sigma(k)}\|_{\ell_\infty}\leq\sqrt{n}/\beta\}$. 
To estimate $\mathbf{A}_k^\sigma[W]$, we adopt an argument analogous to the proof of Eq.(6.7) in \cite{Koike2019dejong}, which is inspired by the proof of \cite[Lemma 2]{DZ2020}. 
Let $(\delta_i)_{i=1}^n$ be a sequence of i.i.d.~Bernoulli variables independent of $X$ and $Y$ with $P(\delta_i=1)=1-P(\delta_i=0)=i/(n+1)$. We set $\zeta_{k,i}:=\delta_k X_i+(1-\delta_k)Y_i$ for all $k,i\in[n]$. 
Then Lemma \ref{cck-derivative} yields
\[
\mathbf{A}_k^\sigma[W]
\lesssim n^{-\frac{m}{2}}\sum_{j_1,\dots,j_m=1}^d\ex{\labs\Upsilon^{j_1,\dots,j_m}\lpa\wh{S}_n^\sigma(k)+n^{-\frac{1}{2}}\zeta_{k,{\sigma(k)}}-y\rpa\rabs} \expe{|W_{\sigma(k),j_1}|^m}.
\]
Next, for any $k,i\in[n]$, we set
\[
\mathcal{A}_{k,i}=\{(A,B):A\subset[n],B\subset[n],A\cup B=[n]\setminus\{i\},\#A=k-1,\#B=n-k\},
\]
where $\#S$ denotes the number of elements in a set $S$. 
We also set
\[
\mathcal{A}_{k}=\{(A,B):A\subset[n],B\subset[n],A\cup B=[n],\#A=k,\#B=n-k\}
\]
for every $k\in\{0,1\dots,n\}$. 
Moreover, for any $(A,B)\in\bigcup_{k=0}^n\mathcal{A}_k$ and $k\in[n]$, we set
\[
W^{(A,B)}_k=
\left\{
\begin{array}{cl}
X_k & \text{if }k\in A,     \\
Y_k  & \text{if }k\in B.  
\end{array}
\right.
\]
Then we define $S_n^{(A,B)}:=\sum_{k=1}^nW^{(A,B)}_{k}$ and $\wh{S}_n^{(A,B)}(i):=\sum_{k=1,k\neq i}^nW^{(A,B)}_{k}$ for $i=1,\dots,n$. 
Finally, for any $\sigma\in\mathfrak{S}_n$ and $k\in[n]$ we set $A^\sigma_k:=\{\sigma(1),\dots,\sigma(k-1)\}$ and $B^\sigma_k:=\{\sigma(k+1),\dots,\sigma(n)\}$. 
Now, since $\wh{S}_n^\sigma(k)=\wh{S}_n^{(A^\sigma_k,B^\sigma_k)}(\sigma(k))$ for every $k\in[n]$, we obtain
\begin{align*}
&\frac{1}{n!}\sum_{\sigma\in\mathfrak{S}_n}\sum_{k=1}^n\mathbf{A}_k^\sigma[W]\\
&\lesssim\frac{n^{-\frac{m}{2}}}{n!}\sum_{\sigma\in\mathfrak{S}_n}\sum_{k=1}^n\sum_{j_1,\dots,j_m=1}^d\ex{\labs\Upsilon^{j_1,\dots,j_m}\lpa\wh{S}_n^{(A^\sigma_k,B^\sigma_k)}(\sigma(k))+n^{-\frac{1}{2}}\zeta_{k,{\sigma(k)}}-y\rpa\rabs} \expe{|W_{\sigma(k),j_1}|^m}\\
&=\frac{n^{-\frac{m}{2}}}{n!}\sum_{k=1}^n\sum_{i=1}^n\sum_{\sigma\in\mathfrak{S}_n:\sigma(k)=i}\sum_{j_1,\dots,j_m=1}^d\ex{\labs\Upsilon^{j_1,\dots,j_m}\lpa\wh{S}_n^{(A^\sigma_k,B^\sigma_k)}(i)+n^{-\frac{1}{2}}\zeta_{k,i}-y\rpa\rabs} \expe{|W_{i,j_1}|^m}\\
&=n^{-\frac{m}{2}}\sum_{k=1}^n\frac{(k-1)!(n-k)!}{n!}\sum_{i=1}^n\sum_{(A,B)\in\mcl{A}_{k,i}}\sum_{j_1,\dots,j_m=1}^d\ex{\labs\Upsilon^{j_1,\dots,j_m}\lpa\wh{S}_n^{(A,B)}(i)+n^{-\frac{1}{2}}\zeta_{k,i}-y\rpa\rabs} \expe{|W_{i,j_1}|^m},
\end{align*}
where we use the identity $\#\{\sigma\in\mf{S}_n:A^\sigma_k=A,\sigma(k)=i\}=(k-1)!(n-k)!$ to deduce the last equality. 
Now, for $(A,B)\in\mcl{A}_{k,i}$ we have $\wh{S}_n^{(A,B)}(i)+n^{-\frac{1}{2}}X_i=S_n^{(A\cup\{i\},B)}$ and $\wh{S}_n^{(A,B)}(i)+n^{-\frac{1}{2}}Y_i=S_n^{(A,B\cup\{i\})}$, so we obtain
\begin{align*}
&\frac{1}{n!}\sum_{\sigma\in\mathfrak{S}_n}\sum_{k=1}^n\mathbf{A}_k^\sigma[W]\\
&\lesssim n^{-\frac{m}{2}}\sum_{k=1}^n\frac{(k-1)!(n-k)!}{n!}\sum_{i=1}^n\sum_{(A,B)\in\mcl{A}_{k,i}}\sum_{j_1,\dots,j_m=1}^d\left\{\frac{k}{n+1}\ex{\labs\Upsilon^{j_1,\dots,j_m}\lpa\wh{S}_n^{(A\cup\{i\},B)}-y\rpa\rabs} \expe{|W_{i,j_1}|^m}\right.\\
&\left.\hphantom{\lesssim n^{-\frac{m}{2}}\sum_{k=1}^n\frac{(k-1)!(n-k)!}{n!}}+\frac{n+1-k}{n+1}\ex{\labs\Upsilon^{j_1,\dots,j_m}\lpa\wh{S}_n^{(A,B\cup\{i\})}-y\rpa\rabs} \expe{|W_{i,j_1}|^m}\right\}\\
&=n^{-\frac{m}{2}}\sum_{k=1}^n\frac{k!(n-k)!}{(n+1)!}\sum_{i=1}^n\sum_{(A,B)\in\mcl{A}_{k,i}}\sum_{j_1,\dots,j_m=1}^d\ex{\labs\Upsilon^{j_1,\dots,j_m}\lpa\wh{S}_n^{(A\cup\{i\},B)}-y\rpa\rabs} \expe{|W_{i,j_1}|^m}\\
&\quad+n^{-\frac{m}{2}}\sum_{k=1}^n\frac{(k-1)!(n+1-k)!}{(n+1)!}\sum_{i=1}^n\sum_{(A,B)\in\mcl{A}_{k,i}}\sum_{j_1,\dots,j_m=1}\ex{\labs\Upsilon^{j_1,\dots,j_m}\lpa\wh{S}_n^{(A,B\cup\{i\})}-y\rpa\rabs} \expe{|W_{i,j_1}|^m}\\
&=n^{-\frac{m}{2}}\sum_{k=1}^n\frac{k!(n-k)!}{(n+1)!}\sum_{i=1}^n\sum_{(A,B)\in\mcl{A}_{k,i}}\sum_{j_1,\dots,j_m=1}^d\ex{\labs\Upsilon^{j_1,\dots,j_m}\lpa\wh{S}_n^{(A\cup\{i\},B)}-y\rpa\rabs} \expe{|W_{i,j_1}|^m}\\
&\quad+n^{-\frac{m}{2}}\sum_{k=0}^{n-1}\frac{k!(n-k)!}{(n+1)!}\sum_{i=1}^n\sum_{(A,B)\in\mcl{A}_{k+1,i}}\sum_{j_1,\dots,j_m=1}\ex{\labs\Upsilon^{j_1,\dots,j_m}\lpa\wh{S}_n^{(A,B\cup\{i\})}-y\rpa\rabs} \expe{|W_{i,j_1}|^m}\\
&=n^{-\frac{m}{2}}\sum_{k=1}^n\frac{k!(n-k)!}{(n+1)!}\sum_{i=1}^n\sum_{(A,B)\in\mcl{A}_{k}:i\in A}\sum_{j_1,\dots,j_m=1}^d\ex{\labs\Upsilon^{j_1,\dots,j_m}\lpa\wh{S}_n^{(A,B)}-y\rpa\rabs} \expe{|W_{i,j_1}|^m}\\
&\quad+n^{-\frac{m}{2}}\sum_{k=0}^{n-1}\frac{k!(n-k)!}{(n+1)!}\sum_{i=1}^n\sum_{(A,B)\in\mcl{A}_{k}:i\in B}\sum_{j_1,\dots,j_m=1}\ex{\labs\Upsilon^{j_1,\dots,j_m}\lpa\wh{S}_n^{(A,B)}-y\rpa\rabs} \expe{|W_{i,j_1}|^m}\\
&=n^{-\frac{m}{2}}\sum_{k=0}^n\frac{k!(n-k)!}{(n+1)!}\sum_{i=1}^n\sum_{(A,B)\in\mcl{A}_{k}}\sum_{j_1,\dots,j_m=1}^d\ex{\labs\Upsilon^{j_1,\dots,j_m}\lpa\wh{S}_n^{(A,B)}-y\rpa\rabs} \expe{|W_{i,j_1}|^m}\\
&=n^{-\frac{m}{2}}\sum_{k=0}^n\frac{k!(n-k)!}{(n+1)!}\sum_{(A,B)\in\mcl{A}_{k}}\sum_{j_1,\dots,j_m=1}^d\ex{\labs\Upsilon^{j_1,\dots,j_m}\lpa\wh{S}_n^{(A,B)}-y\rpa\rabs}\sum_{i=1}^n\expe{|W_{i,j_1}|^m}\\
&\leq n^{-\frac{m}{2}}\sum_{k=0}^n\frac{k!(n-k)!}{(n+1)!}\sum_{(A,B)\in\mcl{A}_{k}}\sum_{j_1,\dots,j_m=1}^d\ex{\labs\Upsilon^{j_1,\dots,j_m}\lpa\wh{S}_n^{(A,B)}-y\rpa\rabs} \max_{1\leq j\leq d}\sum_{i=1}^n\expe{|W_{i,j}|^m}.
\end{align*}
Therefore, noting $\#\mcl{A}_k=n!/\{k!(n-k)!\}$, we conclude by Lemma \ref{cck-derivative} that
\begin{equation}\label{est-A}
\frac{1}{n!}\sum_{\sigma\in\mathfrak{S}_n}\sum_{k=1}^n\mathbf{A}_k^\sigma[W]
\lesssim_m n^{-\frac{m}{2}}\max_{1\leq l\leq m}\beta^{m-l}\|h^{(l)}\|_\infty\max_{1\leq j\leq d}\sum_{i=1}^n\expe{|W_{i,j}|^m}.
\end{equation}
Next we estimate $\mathbf{B}_k^\sigma[W]$. Using the independence of $X_{\sigma(k)}$ and $Y_{\sigma(k)}$ from $\wh{S}_n^\sigma(k)$ as well as Lemma \ref{cck-derivative}, we obtain
\begin{align*}
\mathbf{B}_k^\sigma[W]
&\leq n^{-\frac{m}{2}}\sum_{j_1,\dots,j_m=1}^d\ex{\labs\Upsilon^{j_1,\dots,j_m}\lpa\wh{S}_n^\sigma(k)-y\rpa\rabs}P\lpa\mcl{E}_{\sigma,k}^c\rpa \max_{1\leq j\leq d}\expe{|W_{\sigma(k),j}|^m}\\
&\lesssim_m n^{-\frac{m}{2}}\max_{1\leq l\leq m}\beta^{m-l}\|h^{(l)}\|_\infty P\lpa\mcl{E}_{\sigma,k}^c\rpa \max_{1\leq j\leq d}\expe{|W_{\sigma(k),j}|^m}.
\end{align*}
Thus we conclude that
\begin{equation}\label{est-B}
\frac{1}{n!}\sum_{\sigma\in\mathfrak{S}_n}\sum_{k=1}^n\mathbf{B}_k^\sigma[W]
\lesssim_m n^{-\frac{m}{2}}\max_{1\leq l\leq m}\beta^{m-l}\|h^{(l)}\|_\infty \sum_{i=1}^nP\lpa\|X_{i}\|_{\ell_\infty}\vee\|Y_{i}\|_{\ell_\infty}>\sqrt{n}/\beta\rpa \max_{1\leq j\leq d}\expe{|W_{i,j}|^m}.
\end{equation}
Combining \eqref{est-I-pre} with \eqref{est-A}--\eqref{est-B}, we obtain
\begin{multline}\label{est-I}
\frac{1}{n!}\sum_{\sigma\in\mathfrak{S}_n}\sum_{k=1}^n\mathbf{I}_k^\sigma[W]
\lesssim_m n^{-\frac{m}{2}}\lpa\max_{1\leq l\leq m}\beta^{m-l}\|h^{(l)}\|_\infty\rpa\left\{ 
\max_{1\leq j\leq d}\sum_{i=1}^n\expe{|W_{i,j}|^m}\right.\\
\left.+\sum_{i=1}^nP\lpa\|X_{i}\|_{\ell_\infty}\vee\|Y_{i}\|_{\ell_\infty}>\sqrt{n}/\beta\rpa \max_{1\leq j\leq d}\expe{|W_{i,j}|^m}
\right\}.
\end{multline}

Now, combining \eqref{lindeberg-eq1}--\eqref{lindeberg-eq2} with \eqref{est-R}--\eqref{est-II} and \eqref{est-I}, we obtain the desired result.
\end{proof}

\subsection{Stein kernel}\label{sec:stein}

\begin{definition}[Stein kernel]
Let $F$ be a centered random vector in $\mathbb{R}^d$. 
A $d\times d$ matrix-valued measurable function $\tau_F=(\tau_F^{ij})_{1\leq i,j\leq d}$ on $\mathbb{R}^d$ is called a \textit{Stein kernel} for (the law of) $F$ if $\expe{|\tau_F^{ij}(F)|}<\infty$ for any $i,j\in[d]$ and 
\begin{equation*}
\sum_{j=1}^d\expe{\partial_j\varphi(F)F_j}=\sum_{i,j=1}^d\expe{\partial_{ij}\varphi(F)\tau_F^{ij}(F)}
\end{equation*}
for any $\varphi\in C^\infty_b(\mathbb{R}^d)$.
\end{definition}
When a random vector $F$ has a Stein kernel, it serves for obtaining a good upper bound of $\rho_{h,\beta}(F,Z)$ for a Gaussian vector $Z$. This is formally developed in Section 4 of \cite{Koike2019dejong} with inspired by arguments in \cite{CCK2015} and \cite{Koike2017stein}:
\begin{lemma}[\cite{Koike2019dejong}, Lemma 4.1]\label{lemma:stein}
Let $F$ and $Z$ be centered random vectors in $\mathbb{R}^d$. 
Assume $Z$ is Gaussian. 
Assume also that $F$ has a Stein kernel $\tau_F=(\tau_F^{ij})_{1\leq i,j\leq d}$. Then we have
\[
\rho_{h,\beta}(F,Z)
\leq \frac{3}{2}\max\{\|h''\|_\infty,\beta\|h'\|_\infty\}\ex{\max_{1\leq i,j\leq d}|\tau_F^{ij}(F)-\expe{Z_iZ_j}|}
\]
for any $\beta>0$ and $h\in C^\infty_b(\mathbb{R})$.
\end{lemma}

The following simple lemma plays a key role in our arguments.  
\begin{lemma}\label{lemma:uni-stein}
Let $\xi=(\xi_i)_{i=1}^n$ be a sequence of independent centered random variables with unit variance and $A=(a_{ij})_{1\leq i\leq n,1\leq j\leq d}$ be a $n\times d$ matrix. 
Define the $d$-dimensional random vector $F$ by
\[
F_j=\sum_{i=1}^na_{ij}\xi_i,\qquad j=1,\dots,d.
\]
Suppose that $\xi_i$ has a stein kernel $\tau_i$ for every $i$ and define the $d\times d$ matrix-valued function $\tau_F=(\tau_F^{jk})_{1\leq j,k\leq d}$ on $\mathbb{R}^d$ by
\[
\tau_F^{jk}(x)=\ex{\sum_{i=1}^na_{ij}a_{ik}\tau_i(\xi_i)\mid F=x},\qquad x\in\mathbb{R}^d.
\]
Then $\tau_F$ is a stein kernel for $F$. Moreover, it holds that
\begin{equation}\label{stein-bound}
\ex{\max_{1\leq j,k\leq d}\labs\tau_F^{jk}(F)-\sum_{i=1}^na_{ij}a_{ik}\rabs}
\leq\sqrt{2\log(2d^2)}\max_{1\leq j\leq d}\sqrt{\sum_{i=1}^na_{ij}^4(\|\tau_i\|_\infty+1)^2}.
\end{equation}
\end{lemma}

\begin{proof}
First we show that $\tau_F$ is a Stein kernel for $F$. 
Take $\varphi\in C_b^\infty(\mathbb{R}^d)$ arbitrarily. For every $j\in[d]$, we define the function $f_j:\mathbb{R}^n\to\mathbb{R}$ by
\[
f_j(x_1,\dots,x_n)=\partial_j\varphi\left(\sum_{i=1}^na_{i1}x_i,\dots,\sum_{i=1}^na_{id}x_i\right),\qquad
(x_1,\dots,x_n)\in\mathbb{R}^n.
\]
Also, we denote by $\mcl{L}_i$ the law of $\xi_i$ for every $i\in[n]$. Then we have
\begin{align*}
\sum_{j=1}^d\expe{\partial_j\varphi(F)F_j}
&=\sum_{j=1}^d\sum_{i=1}^na_{ij}\expe{f_j(\xi_1,\dots,\xi_n)\xi_i}\\
&=\sum_{j=1}^d\sum_{i=1}^na_{ij}\int_{\mathbb{R}^n}f_j(x_1,\dots,x_n)x_i\mcl{L}_1(dx_1)\cdots\mcl{L}_n(dx_n)~(\because\text{independence of $\xi$})\\
&=\sum_{j=1}^d\sum_{i=1}^na_{ij}\int_{\mathbb{R}^n}\frac{\partial f_j}{\partial x_i}(x_1,\dots,x_n)\tau_i(x_i)\mcl{L}_1(dx_1)\cdots\mcl{L}_n(dx_n)~(\because\text{definition of $\tau_i$})\\
&=\sum_{j=1}^d\sum_{i=1}^na_{ij}\ex{\sum_{k=1}^da_{ik}\partial_{jk}\varphi(F)\tau_i(\xi_i)}
=\sum_{j,k=1}^d\expe{\partial_{jk}\varphi(F)\tau_F^{jk}(F)}.
\end{align*}
This implies that $\tau_F$ is a Stein kernel for $F$.

Next we prove \eqref{stein-bound}. It suffices to consider the case $\max_{1\leq i\leq n}\|\tau_i\|_\infty<\infty$. 
Then, since $\tau_1(\xi_1),\dots,\tau_n(\xi_n)$ are independent, Hoeffding's inequality \cite[Lemma 14.14]{BvdG2011} yields
\begin{align*}
&\ex{\max_{1\leq j,k\leq d}\labs\tau_F^{jk}(F)-\sum_{i=1}^na_{ij}a_{ik}\rabs}
=\ex{\max_{1\leq j,k\leq d}\labs\sum_{i=1}^na_{ij}a_{ik}(\tau_i(\xi_i)-1)\rabs}\\
&\leq\sqrt{2\log(2d^2)}\max_{1\leq j,k\leq d}\sqrt{\sum_{i=1}^na_{ij}^2a_{ik}^2(\|\tau_i\|_\infty+1)^2}
=\sqrt{2\log(2d^2)}\max_{1\leq j\leq d}\sqrt{\sum_{i=1}^na_{ij}^4(\|\tau_i\|_\infty+1)^2}.
\end{align*}
This completes the proof.
\end{proof}

The following estimate is a simple consequence of Nemirovski's inequality. 
\begin{lemma}\label{lemma:nemirovski}
Under the assumptions of Lemma \ref{coupling}, 
\begin{align*}
\ex{\max_{1\leq j,k\leq d}\frac{1}{n}\labs\sum_{i=1}^n(X_{ij}X_{ik}-\expe{X_{ij}X_{ik}})\rabs}
\leq\sqrt{\frac{8\log(2d^2)}{n}}\Delta^X_{n,1}.
\end{align*}
\end{lemma}

\begin{proof}
Nemirovski's inequality \cite[Lemma 14.24]{BvdG2011} implies that
\begin{align*}
\ex{\max_{1\leq j,k\leq d}\labs\sum_{i=1}^n(X_{ij}X_{ik}-\expe{X_{ij}X_{ik}})\rabs}
\leq\sqrt{8\log(2d^2)}\ex{\max_{1\leq j\leq d}\sqrt{\sum_{i=1}^nX_{ij}^4}}.
\end{align*}
Now the desired result follows from the Lyapunov inequality.  
\end{proof}

Now, we can prove Proposition \ref{prop:beta} using the results established in this subsection. 
\begin{proof}[Proof of Proposition \ref{prop:beta}]
By \cite[Example 4.9(c)]{LRS2017}, $\eta_i$ has the Stein kernel $\tau_i^0(x):=2^{-1}x(1-x)1_{(0,1)}(x)$, $x\in\mathbb{R}$. Then, a simple computation shows that $\xi_i$ has the Stein kernel $\tau_i(x):=16\tau_i^0((x+1)/4)=2^{-1}(x+1)(3-x)1_{(-1,3)}(x)$, $x\in\mathbb{R}$. Therefore, Lemmas \ref{lemma:stein}--\ref{lemma:uni-stein} yield
\begin{align*}
&\ex{\sup_{y\in\mathbb{R}^d}\left|\ex{h\left(\Phi_\beta(S_n^Y-y)\right)\mid X}-\ex{h\left(\Phi_\beta(Z-y)\right)}\right|}\nonumber\\
&\leq \frac{3}{2}\lpa\max_{1\leq l\leq 2}\beta^{2-l}\|h^{(l)}\|_\infty\rpa\left\{\ex{\max_{1\leq j,k\leq d}\labs\frac{1}{n}\sum_{i=1}^nX_{ij}X_{ik}-\mf{C}_{jk}\rabs}
+\frac{3\sqrt{2\log(2d^2)}}{2n}\ex{\max_{1\leq j\leq d}\sqrt{\sum_{i=1}^n|X_{ij}|^4}}
\right\}.
\end{align*}
Hence, Lemma \ref{lemma:nemirovski} and the Lyapunov inequality imply that
\begin{align*}
\rho_{h,\beta}(S_n^Y,Z)
&\lesssim \lpa\max_{1\leq l\leq 2}\beta^{2-l}\|h^{(l)}\|_\infty\rpa\left\{
\Delta^X_{n,0}
+\sqrt{\frac{\log d}{n}}\Delta^X_{n,1}
\right\}.
\end{align*}
This completes the proof. 
\end{proof}

\subsection{Proof of Lemma \ref{coupling}}

\begin{lemma}\label{lemma:main}
Under the assumptions of Lemma \ref{coupling}, there is a universal constant $C>0$ such that
\begin{multline*}
\rho_{h,\beta}(S_n^X,Z)
\leq C\left[\lpa\max_{1\leq l\leq 2}\beta^{2-l}\|h^{(l)}\|_\infty\rpa\lpa\Delta_{n,0}^X+\Delta_{n,1}^X\sqrt{\frac{\log d}{n}}\rpa\right.\\
\left.\qquad+\lpa\max_{1\leq l\leq 4}\beta^{4-l}\|h^{(l)}\|_\infty\rpa\frac{\lpa\Delta_{n,1}^X\rpa^2
+\Delta_{n,2}^X(\beta^{-1}\log d)}{n}
\right]
\end{multline*}
for any $h\in C_b^\infty(\mathbb{R})$ and $\beta>0$.
\end{lemma}

\begin{proof}
Let us define $(\xi_i)_{i=1}^n$ and $Y=(Y_i)_{i=1}^n$ as in Proposition \ref{prop:beta}. 
A straightforward computation shows $\expe{\xi_i}=0$ and $\expe{\xi_i^2}=\expe{\xi_i^3}=1$, so we can apply Lemma \ref{lemma:lindeberg} to $X$ and $Y$ with $m=4$.
Then, noting that $|\xi_i|\leq3$, we obtain
\begin{align*}
&\rho_{h,\beta}(S_n^X,S_n^Y)\\
&\lesssim \frac{1}{n^2}\lpa\max_{1\leq l\leq 4}\beta^{4-l}\|h^{(l)}\|_\infty\rpa\left\{\max_{1\leq j\leq d}\sum_{i=1}^n\expe{|X_{ij}|^4}
+\sum_{i=1}^n\expe{\|X_{i}\|_{\ell_\infty}^4;\|X_{i}\|_{\ell_\infty}>\sqrt{n}/(3\beta)}
\right.\\
&\left.\hphantom{\leq \frac{C_1}{n^2}\lpa\max_{1\leq l\leq 4}\beta^{4-l}\|h^{(l)}\|_\infty\rpa}+\sum_{i=1}^nP\lpa\|X_{i}\|_{\ell_\infty}>\sqrt{n}/(3\beta)\rpa \max_{1\leq j\leq d}\expe{|X_{ij}|^4}\right\}.
\end{align*}
We evidently have
\[
\frac{1}{n^2}\max_{1\leq j\leq d}\sum_{i=1}^n\expe{|X_{ij}|^4}\leq\frac{\lpa\Delta_{n,1}^X\rpa^2}{n}.
\]
Moreover, Chebyshev's association inequality (see e.g.~Theorem 2.14 in \cite{BLM2013}) yields
\begin{align*}
\sum_{i=1}^nP\lpa\|X_{i}\|_{\ell_\infty}>\sqrt{n}/(3\beta)\rpa \max_{1\leq j\leq d}\expe{|X_{ij}|^4}
\leq\sum_{i=1}^n\max_{1\leq j\leq d}\expe{|X_{ij}|^4;\|X_{i}\|_{\ell_\infty}>\sqrt{n}/(3\beta)}.
\end{align*}
So we conclude that
\begin{align}
\rho_{h,\beta}(S_n^X,S_n^Y)
&\lesssim\lpa\max_{1\leq l\leq 4}\beta^{4-l}\|h^{(l)}\|_\infty\rpa\frac{\lpa\Delta_{n,1}^X\rpa^2+\Delta^X_{n,2}(\beta^{-1}\log d)}{n}.\label{eq:match}
\end{align}
Since $\rho_{h,\beta}(S_n^X,Z)\leq \rho_{h,\beta}(S_n^X,S_n^Y)+\rho_{h,\beta}(S_n^Y,Z)$, we obtain the desired result from \eqref{eq:match} and Proposition \ref{prop:beta}. 
\end{proof}

\begin{proof}[Proof of Lemma \ref{coupling}]
Without loss of generality, we may assume 
\begin{equation}\label{eq:trivial0}
\varepsilon^{-2}\Delta^X_{n,1}\sqrt{\frac{(\log d)^3}{n}}\leq1, 
\end{equation}
since otherwise the claim obviously holds true with $C=1$. 

Set $\beta=\varepsilon^{-1}\log d$ (hence $\beta^{-1}\log d=\varepsilon$). By \eqref{max-smooth} we have
\[
P\lpa\max_{1\leq j\leq d}(S_{n,j}^X-y_j)\in A\rpa\leq P(\Phi_\beta(S_n^X-y)\in A^{\varepsilon})=\expe{1_{A^{\varepsilon}}(\Phi_\beta(S_n^X-y))}.
\]
Next, by Lemma \ref{cck-approx} there is a $C^\infty$ function $h:\mathbb{R}\to\mathbb{R}$ and a universal constant $K>0$ such that $\|h^{(r)}\|_\infty\leq K\varepsilon^{-r}$ for $r=1,2,3,4$ and $1_{A^{\varepsilon}}(x)\leq h(x)\leq1_{A^{4\varepsilon}}(x)$ for all $x\in\mathbb{R}$. Then we have $\expe{1_{A^{\varepsilon}}(\Phi_\beta(S_{n}^X-y))}\leq \expe{h(\Phi_\beta(S_{n}^X-y))}.$ 
Now, by Lemma \ref{lemma:main} we have
\begin{align*}
\rho_{h,\beta}(S_n^X,Z)
&\lesssim \varepsilon^{-2}(\log d)\lpa\Delta^X_{n,0}+\Delta^X_{n,1}\sqrt{\frac{\log d}{n}}\rpa
+\varepsilon^{-4}(\log d)^3\frac{\lpa\Delta_{n,1}^X\rpa^2
+\Delta^X_{n,2}(\varepsilon)}{n}.
\end{align*}
Hence, \eqref{eq:trivial0} yields
\begin{align*}
\rho_{h,\beta}(S_n^X,Z)
\lesssim \varepsilon^{-2}(\log d)\lpa\Delta^X_{n,0}+\Delta^X_{n,1}\sqrt{\frac{\log d}{n}}\rpa
+\varepsilon^{-4}\frac{(\log d)^3}{n}\Delta^X_{n,2}(\varepsilon).
\end{align*}
Meanwhile, we also have 
\[
\expe{h(\Phi_\beta(Z-y))}
\leq \expe{1_{A^{4\varepsilon}}(\Phi_\beta(Z-y))}
\leq \ex{1_{A^{5\varepsilon}}\lpa\max_{1\leq j\leq d}(Z_j-y_j)\rpa}=P\lpa\max_{1\leq j\leq d}(Z_j-y_j)\in A^{5\varepsilon}\rpa.
\]
Consequently, we complete the proof.
\end{proof}

\section{Proofs for Section \ref{sec:main}}\label{sec:proof2}

For $d$-dimensional random vector $F$, we define the $2d$-dimensional random vector $F^\diamond$ by
\[
F^\diamond:=(F_1,\dots,F_d,-F_1,\dots,-F_d)^\top.
\] 
Also, for a sequence $X=(X_i)_{i=1}^n$ of random vectors in $\mathbb{R}^d$, we set $X^\diamond:=(X_i^\diamond)_{i=1}^n$. Note that we have $(S_n^X)^\diamond=S_n^{X^\diamond}$. 

\subsection{Preliminary lemmas}

\begin{lemma}\label{lemma:cf-joint}
Let $Z$ be a Gaussian vector in $\mathbb{R}^d$. Then, $\mcl{C}_{Z^\diamond}(\ve)\leq 2\mcl{C}_Z(\varepsilon)$ for any $\ve>0$. 
\end{lemma}

\begin{proof}
Take $y\in\mathbb{R}^{2d}$ arbitrarily and set $U:=\max_{1\leq j\leq d}(Z_j-y_j)$ and $V:=\max_{1\leq j\leq d}(-Z_j-y_{d+j})$. Then we have
\begin{align*}
&P\lpa0\leq\max_{1\leq j\leq 2d}(Z^\diamond_j-y_j)\leq \varepsilon\rpa
=P\lpa\left\{U\vee V\geq 0\right\}\cap\left\{U\vee V\leq \varepsilon\right\}\rpa\\
&\leq P\lpa\left\{U\geq 0\right\}\cap\left\{U\vee V\leq \varepsilon\right\}\rpa
+P\lpa\left\{V\geq 0\right\}\cap\left\{U\vee V\leq \varepsilon\right\}\rpa\\
&\leq P(0\leq U\leq \varepsilon)+P(0\leq V\leq \varepsilon)
\leq\mcl{C}_Z(\varepsilon)+\mcl{C}_{-Z}(\varepsilon).
\end{align*}
Now, since $-Z$ has the same distribution as $Z$, we have $\mcl{C}_{-Z}(\varepsilon)=\mcl{C}_{Z}(\varepsilon)$. This completes the proof. 
\end{proof}

\begin{lemma}\label{cck-kolmogorov}
Let $F,Z$ be two random vectors in $\mathbb{R}^d$ and assume $Z$ is Gaussian. Assume also that there are constants $\varepsilon,\eta>0$ such that
\begin{equation}\label{eq:cck-kol}
P\lpa\max_{1\leq j\leq 2d}(F^\diamond_j-y_j)\in A\rpa\leq P\lpa\max_{1\leq j\leq 2d}(Z^\diamond_j-y_j)\in A^{\varepsilon}\rpa+\eta
\end{equation}
for any $y\in\mathbb{R}^{2d}$ and Borel set $A\subset\mathbb{R}$. Then we have
\[
\sup_{A\in\mathcal{A}^{\mathrm{re}}}|P(F\in A)-P(Z\in A)|
\leq2\mcl{C}_Z(\varepsilon)+\eta.
\]
\end{lemma}

\begin{proof}
Take $y\in\mathbb{R}^{2d}$ arbitrarily. Then we have
\begin{align*}
P(F^\diamond\leq y)
&=P\lpa\max_j(F^\diamond_j-y_j)\leq0\rpa
\leq P\lpa\max_j(Z^\diamond_j-y_j)\leq\varepsilon\rpa+\eta\quad(\because\text{\eqref{eq:cck-kol}})\\
&=P\lpa0<\max_j(Z^\diamond_j-y_j)\leq0\rpa+P\lpa0<\max_j(Z^\diamond_j-y_j)\leq\varepsilon\rpa+\eta\\
&\leq P(Z^\diamond\leq y)+\mcl{C}_{Z^\diamond}(\varepsilon)+\eta.
\end{align*}
Meanwhile, \eqref{eq:cck-kol} yields
\[
P\lpa\max_{1\leq j\leq 2d}(F^\diamond_j-y_j)>0\rpa\leq P\lpa\max_{1\leq j\leq 2d}(Z^\diamond_j-y_j)>-\ve\rpa+\eta,
\]
so we obtain
\[
P\lpa\max_{1\leq j\leq 2d}(Z^\diamond_j-y_j)\leq-\ve\rpa
\leq P\lpa\max_{1\leq j\leq 2d}(F^\diamond_j-y_j)\leq0\rpa+\eta.
\]
Thus we infer that
\begin{align*}
P(Z^\diamond\leq y)
&=P\lpa\max_j(Z^\diamond_j-y_j)\leq-\ve\rpa+P\lpa-\ve<\max_j(Z^\diamond_j-y_j)\leq0\rpa\\
&\leq P\lpa\max_j(F^\diamond_j-y_j)\leq0\rpa+\eta+\mcl{C}_{Z^\diamond}(\varepsilon)
=P(F^\diamond\leq y)+\eta+\mcl{C}_{Z^\diamond}(\varepsilon).
\end{align*}
So we obtain $|P(F^\diamond\leq y)-P(Z^\diamond\leq y)|\leq 2\mcl{C}_{Z}(\varepsilon)+\eta$ by Lemma \ref{lemma:cf-joint}. Since
\[
\sup_{A\in\mathcal{A}^{\mathrm{re}}}|P(F\in A)-P(Z\in A)|
=\sup_{y\in\mathbb{R}^{2d}}|P(F^\diamond\leq y)-P(Z^\diamond\leq y)|,
\]
this completes the proof.
\end{proof}

%

\begin{lemma}\label{lemma:theta}
$\mcl{C}_F(\ve)>0$ for any $d$-dimensional random vector $F$ and $\ve>0$.
\end{lemma}

\begin{proof}
To obtain a contradiction, assume $\mcl{C}_F(\ve)=0$. Then we have $P(x\leq F\leq x+\ve)=0$ for all $x\in\mathbb{R}$. Thus, 
$1=P(F\in\mathbb{R})\leq\sum_{i=-\infty}^\infty P(i\ve\leq F\leq (i+1)\ve)=0,$ 
a contradiction.
\end{proof}

\subsection{Proof of Theorem \ref{thm:main}\ref{thm:psi}}

The following is a generalization of \cite[Lemma C.1]{CCK2017}:
\begin{lemma}\label{lemma:c1}
Let $Y$ be a non-negative random variable such that $P(Y>x)\leq Ae^{-x/B}$ for all $y\geq0$ and for some constants $A,B>0$. Then we have $E[Y^p1_{\{Y>t\}}]\leq p!Ae^{-t/B}(t+B)^p$ for every $t\geq0$ and every positive integer $p$.
\end{lemma}

\begin{proof}
A simple computation yields
\begin{align*}
E[Y^p1_{\{Y>t\}}]
&=p\int_0^tP(Y>t)x^{p-1}dx+p\int_t^\infty P(Y>x)x^{p-1}dx\\
&\leq At^pe^{-t/B}+pA\int_t^\infty e^{-x/B}x^{p-1}dx.
\end{align*}
By Eq.(8.352.2) in \cite{GR2007} we have
\begin{align*}
pA\int_t^\infty x^{p-1}e^{-x/B}dx
=pAB^{p}\int_{t/B}^\infty y^{p-1}e^{-y}dx
=p!AB^{p}e^{-t/B}\sum_{q=0}^{p-1}\frac{(t/B)^q}{q!}.
\end{align*}
Consequently, we obtain
\begin{align*}
E[Y^p1_{\{Y>t\}}]
&\leq p!AB^{p}e^{-t/B}\sum_{q=0}^{p}\frac{(t/B)^q}{q!}
\leq p!Ae^{-t/B}(t+B)^p.
\end{align*}
This completes the proof.
\end{proof}

\begin{lemma}\label{lemma:bounded}
If there are constants $B_n,\kappa_n\geq1$ such that $\max_{j}n^{-1}\sum_{i=1}^n\expe{X_{ij}^4}\leq B_n^2$ and $\max_{i,j}|X_{ij}|\leq 2\kappa_n$ a.s., there is a universal constant $C>0$ such that 
\[
\Delta^X_{n,1}\leq C\lpa B_n+\kappa_n^2\sqrt{\frac{\log d}{n}}\rpa.
\]
\end{lemma}

\begin{proof}
By \cite[Lemma 9]{CCK2015} and assumptions we have
\begin{align*}
\ex{\max_{1\leq j\leq d}\sum_{i=1}^nX_{ij}^4}
&\lesssim\max_{1\leq j\leq d}\sum_{i=1}^n\ex{X_{ij}^4}
+(\log d)\ex{\max_{1\leq i\leq n}\max_{1\leq j\leq d}X_{ij}^4}\\
&\lesssim nB_n^2+\kappa_n^4(\log d),
\end{align*}
which yields the desired result. 
\end{proof}

\begin{lemma}\label{lemma:trunc}
Suppose that $\max_{i,j}\|X_{ij}\|_{\psi_1}\leq B_n$ for some $B_n\geq1$. 
Set $\kappa_n:=2B_n\log n$. For $i=1,\dots,n$ and $j=1,\dots,d$, define 
\[
\wh{X}_{ij}:=X_{ij}1_{\{|X_{ij}|>\kappa_n\}}-\ex{X_{ij}1_{\{|X_{ij}|>\kappa_n\}}}
\] 
and $\wh{X}:=(\wh{X}_i)_{i=1}^n$ with $\wh{X}_i=(\wh{X}_{i1},\dots,\wh{X}_{id})^\top$. Then there is a universal constant $C>0$ such that
\[
\|\|S_n^{\wh{X}}\|_{\ell_\infty}\|_{\psi_1}\leq C\sqrt{\delta_{n,2}}.
\]
\end{lemma}

\begin{proof}
For every $p=2,3,\dots$, we have by Lemma \ref{lemma:c1}
\begin{align*}
\ex{|\wh{X}_{ij}|^p}
&\leq2^{p-1}\ex{|X_{ij}|^p1_{\{|X_{ij}|>\kappa_n\}}}
\leq 2^pp!e^{-\kappa_n/B_n}(\kappa_n+B_n)^p\\
&=\frac{p!}{2}(2\kappa_n+2B_n)^{p-2}\cdot\frac{8(\kappa_n+B_n)^2}{n^2},
\end{align*}
so the Bernstein inequality \cite[Lemma 2.2.11]{VW1996} yields
\[
P\lpa\labs S_{n,j}^{\wh{X}}\rabs>x\rpa\leq2\exp\lpa-\frac{1}{2}\frac{x^2}{8(\kappa_n+B_n)^2/n^2+(2\kappa_n+2B_n)x/\sqrt{n}\}}\rpa
\]
for all $j\in\{1,\dots,d\}$ and $x\geq0$. Therefore, by Lemma 2.2.10 in \cite{VW1996} we obtain
\begin{align*}
\|\|S_n^{\wh{X}}\|_{\ell_\infty}\|_{\psi_1}
&\lesssim \sqrt{\frac{(2\kappa_n+2B_n)^2(\log d)^2}{n}}+ \sqrt{\frac{8(\kappa_n+B_n)^2\log d}{n^2}}
\lesssim \sqrt{\delta_{n,2}}.
\end{align*}
This completes the proof.
\end{proof}

Under the present assumptions, we can reduce Lemma \ref{coupling} to the following form:
\begin{lemma}\label{lemma:psi}
Under the assumptions of Theorem \ref{thm:main}\ref{thm:psi}, there is a universal constant $C_0>0$ such that
\[
P\lpa\max_{1\leq j\leq d}(S_{n,j}^{X}-y_j)\in A\rpa
\leq P\lpa\max_{1\leq j\leq d}(Z_{j}^X-y_j)\in A^{6\varepsilon}\rpa
+C_0\varepsilon^{-2}\lpa\sqrt{\delta_{n,1}}+\delta_{n,2}\rpa,
\]
for any $y\in\mathbb{R}^d$, $A\in\mcl{B}(\mathbb{R})$ and $\varepsilon\geq 12B_n(\log n)(\log d)/\sqrt{n}$. 
\end{lemma}

\begin{proof}
Set $\kappa_n:=2B_n\log n$. For $i=1,\dots,n$ and $j=1,\dots,d$, define 
\[
\wt{X}_{ij}:=X_{ij}1_{\{|X_{ij}|\leq\kappa_n\}}-\ex{X_{ij}1_{\{|X_{ij}|\leq\kappa_n\}}}
\] 
and set $\wt{X}:=(\wt{X}_i)_{i=1}^n$ with $\wt{X}_i=(\wt{X}_{i1},\dots,\wt{X}_{id})^\top$. 
Note that $\max_{i,j}|\wt{X}_{ij}|\leq2\kappa_n$. 
Also, we evidently have
\begin{equation}\label{est:trunc0}
P\lpa\max_{1\leq j\leq d}(S_{n,j}^X-y_j)\in A\rpa
\leq P\lpa\max_{1\leq j\leq d}(S_{n,j}^{\wt{X}}-y_j)\in A^\varepsilon\rpa
+P\lpa\|S_{n}^{X-\wt{X}}\|_{\ell_\infty}\geq\varepsilon\rpa.
\end{equation}
Noting $\expe{X_{ij}}=0$, we have $X_{ij}-\wt{X}_{ij}=X_{ij}1_{\{|X_{ij}|>\kappa_n\}}-\expe{X_{ij}1_{\{|X_{ij}|>\kappa_n\}}}$. Hence, Lemma \ref{lemma:trunc} and the Markov inequality yield
\begin{equation}\label{est:trunc}
P\lpa\|S_{n}^{X-\wt{X}}\|_{\ell_\infty}\geq\varepsilon\rpa
\leq\varepsilon^{-2}\ex{\|S_{n}^{X-\wt{X}}\|_{\ell_\infty}^2}
\lesssim\varepsilon^{-2}\delta_{n,2}.
\end{equation}

Next, applying Lemma \ref{coupling} to $\wt{X}$ with $\mf{C}=\mf{C}_n^X$, we obtain
\begin{multline*}
P\lpa\max_{1\leq j\leq d}(S_{n,j}^{\wt{X}}-y_j)\in A^\varepsilon\rpa
\leq P\lpa\max_{1\leq j\leq d}(Z_{n,j}^X-y_j)\in A^{6\varepsilon}\rpa\\
+C\left\{\varepsilon^{-2}\lpa\Delta^{\wt{X}}_{n,0}\log d+\Delta^{\wt{X}}_{n,1}\sqrt{\frac{(\log d)^3}{n}}\rpa
+\varepsilon^{-4}\Delta^{\wt{X}}_{n,2}(\varepsilon)\frac{(\log d)^3}{n}\right\},
\end{multline*}
where $C>0$ is a universal constant. 
The Schwarz inequality and Lemma \ref{lemma:c1} yield
\begin{align}
\Delta^{\wt{X}}_{n,0}&=\max_{1\leq j,k\leq d}\labs\frac{1}{n}\sum_{i=1}^n\lpa\expe{\wt{X}_{ij}\wt{X}_{ik}}-\expe{X_{ij}X_{ik}}\rpa\rabs\nonumber\\
&\leq\max_{1\leq j,k\leq d}\max_{1\leq i\leq n}\lpa\|X_{ij}-\wt{X}_{ij}\|_2\|\wt{X}_{ik}\|_2+\|X_{ij}\|_2\|X_{ik}-\wt{X}_{ik}\|_2\rpa\nonumber\\
&\lesssim e^{-\kappa_n/(2B_n)}B_n^2\log n
\lesssim B_n^2(\log d)(\log n)^2/n.\label{est:Delta0}
\end{align}
Meanwhile, applying Lemma \ref{lemma:bounded} to $\wt{X}$ (note that $\expe{\wt{X}_{ij}^4}\lesssim\expe{X_{ij}^4}$), we obtain
\begin{align*}
\Delta^{\wt{X}}_{n,1}\sqrt{\frac{(\log d)^3}{n}}
\lesssim B_n\sqrt{\frac{(\log d)^3}{n}}+\kappa_n^2\frac{(\log d)^2}{n}
\lesssim \sqrt{\delta_{n,1}}+\delta_{n,2}.
\end{align*}
Moreover, since $\sqrt{n}\varepsilon/(3\log d)\geq2\kappa_n$ by assumption, we have $\Delta^{\wt{X}}_{n,2}(\varepsilon)=0$. 
Consequently, we obtain
\[
P\lpa\max_{1\leq j\leq d}(S_{n,j}^{\wt{X}}-y_j)\in A^\varepsilon\rpa
\leq P\lpa\max_{1\leq j\leq d}(Z_{j}^X-y_j)\in A^{6\varepsilon}\rpa
+C'\varepsilon^{-2}\lpa\sqrt{\delta_{n,1}}+\delta_{n,2}\rpa,
\]
where $C'>0$ is a universal constant. Combining this with \eqref{est:trunc0}--\eqref{est:trunc}, we complete the proof. 
\end{proof}

\begin{proof}[Proof of Theorem \ref{thm:main}\ref{thm:psi}]
Without loss of generality, we may assume
\begin{equation}\label{eq:wlog}
\Theta_X^{2/3}\lpa\delta_{n,1}^{1/6}+\delta_{n,2}^{1/3}\rpa\leq1.
\end{equation}
since otherwise the claim holds true with $C=1$. 
Noting $\Theta_X>0$ by Lemma \ref{lemma:theta}, we set 
\[
\varepsilon:=12\Theta_X^{-1/3}\lpa\delta_{n,1}^{1/6}+\delta_{n,2}^{1/3}\rpa.
\]
Then it holds that $\varepsilon\geq12B_n(\log n)(\log d)/\sqrt{n}.$ 
In fact, 
\begin{align*}
\frac{\sqrt{n}\varepsilon}{12B_n(\log n)(\log d)}
&\geq\Theta_X^{-1/3}\frac{\sqrt{n}}{B_n(\log n)(\log d)}\cdot\frac{B_n^{2/3}(\log d)^{2/3}(\log n)^{2/3}}{n^{1/3}}\\
&=\Theta_X^{-1/3}\frac{n^{1/6}}{B_n^{1/3}(\log n)^{1/3}(\log d)^{1/3}}
=\lpa\Theta_X^{2/3}\delta_{n,2}^{1/3}\rpa^{-1/2},
\end{align*}
so \eqref{eq:wlog} implies the desired inequality. 
Now, since the assumptions of Theorem \ref{thm:main}\ref{thm:psi} are satisfied with replacing $X$ by $X^\diamond$, we can apply Lemma \ref{lemma:psi} to $X^\diamond$ instead of $X$. Thus, there is a universal constant $C_0>0$ such that
\begin{align*}
P\lpa\max_{1\leq j\leq 2d}(S_{n,j}^{X^\diamond}-y_j)\in A\rpa
\leq P\lpa\max_{1\leq j\leq 2d}(Z_{j}^{X^\diamond}-y_j)\in A^{6\varepsilon}\rpa
+C_0\varepsilon^{-2}\lpa\sqrt{\delta_{n,1}}+\delta_{n,2}\rpa
\end{align*}
for any $y\in\mathbb{R}^{2d}$ and $A\in\mcl{B}(\mathbb{R})$. 
Noting $S_n^{X^\diamond}=(S_n^X)^\diamond$, by Lemma \ref{cck-kolmogorov} we obtain
\begin{align*}
\rho_n(\mcl{A}^{\mathrm{re}})
&\leq2\mcl{C}_{Z_n^X}(6\varepsilon)+C_0\varepsilon^{-2}\lpa\sqrt{\delta_{n,1}}+\delta_{n,2}\rpa\\
&\leq12\Theta_X\varepsilon+C_0\varepsilon^{-2}\lpa\sqrt{\delta_{n,1}}+\delta_{n,2}\rpa
\lesssim\Theta_X^{2/3}\lpa\delta_{n,1}^{1/6}+\delta_{n,2}^{1/3}\rpa.
\end{align*}
Thus we complete the proof.
\end{proof}


\subsection{Proof of Theorem \ref{thm:main}\ref{thm:mom}}

In the current situation, Lemma \ref{coupling} can be reduced to the following form:
\begin{lemma}\label{lemma:mom}
Under the assumptions of Theorem \ref{thm:main}\ref{thm:mom}, there is a universal constant $C_0>0$ such that
\[
P\lpa\max_{1\leq j\leq d}(S_{n,j}^{X}-y_j)\in A\rpa
\leq P\lpa\max_{1\leq j\leq d}(Z_{j}^X-y_j)\in A^{6\varepsilon}\rpa
+C_0\varepsilon^{-2}\lpa\sqrt{\delta_{n,1}}+\delta_{n,2}(q)\rpa,
\]
for any $y\in\mathbb{R}^d$, $A\in\mcl{B}(\mathbb{R})$ and $\varepsilon\geq 6D_n(\log d)^{1-1/q}/n^{1/2-1/q}$. 
\end{lemma}

\begin{proof}
The proof is parallel to that of Lemma \ref{lemma:psi}. 
Set $\kappa_n:=D_n(n/\log d)^{1/q}$ so that
\[
\kappa_n^2\frac{(\log d)^2}{n}=(\log d)\frac{D_n^q}{\kappa_n^{q-2}}=\delta_{n,2}(q).
\]
For $i=1,\dots,n$ and $j=1,\dots,d$, define 
\[
\wt{X}_{ij}:=X_{ij}1_{\{\|X_{i}\|_{\ell_\infty}\leq\kappa_n\}}-\ex{X_{ij}1_{\{\|X_{ij}\|_{\ell_\infty}\leq\kappa_n\}}}
\] 
and set $\wt{X}:=(\wt{X}_i)_{i=1}^n$ with $\wt{X}_i=(\wt{X}_{i1},\dots,\wt{X}_{id})^\top$. 
Note that $\max_{i,j}|\wt{X}_{ij}|\leq2\kappa_n$. 
Also, we evidently have \eqref{est:trunc0} with the present notation. 
Moreover, noting $\expe{X_{ij}}=0$, we have $X_{ij}-\wt{X}_{ij}=X_{ij}1_{\{\|X_{i}\|_{\ell_\infty}>\kappa_n\}}-\expe{X_{ij}1_{\{\|X_{i}\|_{\ell_\infty}>\kappa_n\}}}$. 
Thus, Nemirovski's inequality and assumptions yield
\begin{align*}
\ex{\|S_{n}^{X-\wt{X}}\|_{\ell_\infty}^2}
\lesssim\frac{\log d}{n}\ex{\max_{1\leq j\leq d}\sum_{i=1}^nX_{ij}^21_{\{\|X_{i}\|_{\ell_\infty}>\kappa_n\}}}
\leq(\log d)\frac{D_n^q}{\kappa_n^{q-2}}=\delta_{n,2}(q).
\end{align*}
Hence the Markov inequality yield
\begin{equation}\label{est:trunc-mom}
P\lpa\|S_{n}^{X-\wt{X}}\|_{\ell_\infty}\geq\varepsilon\rpa
\lesssim\varepsilon^{-2}\delta_{n,2}(q).
\end{equation}

Next, applying Lemma \ref{coupling} to $\wt{X}$ with $\mf{C}=\mf{C}_n^X$, we obtain
\begin{multline*}
P\lpa\max_{1\leq j\leq d}(S_{n,j}^{\wt{X}}-y_j)\in A^\varepsilon\rpa
\leq P\lpa\max_{1\leq j\leq d}(Z_{n,j}^X-y_j)\in A^{6\varepsilon}\rpa\\
+C\left\{\varepsilon^{-2}\lpa\Delta^{\wt{X}}_{n,0}\log d+\Delta^{\wt{X}}_{n,1}\sqrt{\frac{(\log d)^3}{n}}\rpa
+\varepsilon^{-4}\Delta^{\wt{X}}_{n,2}(\varepsilon)\frac{(\log d)^3}{n}\right\},
\end{multline*}
where $C>0$ is a universal constant. 
Noting $\expe{X_{ij}}=0$, we have
\begin{align}
\Delta^{\wt{X}}_{n,0}&=\max_{1\leq j,k\leq d}\labs\frac{1}{n}\sum_{i=1}^n\lpa\expe{\wt{X}_{ij}\wt{X}_{ik}}-\expe{X_{ij}X_{ik}}\rpa\rabs\nonumber\\
&\leq2\max_{1\leq j\leq d}\max_{1\leq i\leq n}\ex{X_{ij}^21_{\{\|X_i\|_{\ell_\infty}>\kappa_n\}}}
\leq \frac{D_n^q}{\kappa_n^{q-2}}.\label{est:Delta0-mom}
\end{align}
Meanwhile, applying Lemma \ref{lemma:bounded} to $\wt{X}$ (note that $\expe{\wt{X}_{ij}^4}\lesssim\expe{X_{ij}^4}$), we obtain
\begin{align*}
\Delta^{\wt{X}}_{n,1}\sqrt{\frac{(\log d)^3}{n}}
\lesssim B_n\sqrt{\frac{(\log d)^3}{n}}+\kappa_n^2\frac{(\log d)^2}{n}
=\sqrt{\delta_{n,1}}+\delta_{n,2}(q).
\end{align*}
Moreover, since $\sqrt{n}\varepsilon/(3\log d)\geq2\kappa_n$ by assumption, we have $\Delta^{\wt{X}}_{n,2}(\varepsilon)=0$. 
Consequently, we obtain
\[
P\lpa\max_{1\leq j\leq d}(S_{n,j}^{\wt{X}}-y_j)\in A^\varepsilon\rpa
\leq P\lpa\max_{1\leq j\leq d}(Z_{j}^X-y_j)\in A^{6\varepsilon}\rpa
+C'\varepsilon^{-2}\lpa\sqrt{\delta_{n,1}}+\delta_{n,2}(q)\rpa,
\]
where $C'>0$ is a universal constant. Combining this with \eqref{est:trunc0} and \eqref{est:trunc-mom}, we complete the proof. 
\end{proof}

\begin{proof}[Proof of Theorem \ref{thm:main}\ref{thm:mom}]
Without loss of generality, we may assume
\begin{equation}\label{eq:wlog-mom}
\Theta_X^{2/3}\lpa\delta_{n,1}^{1/6}+\delta_{n,2}(q)^{1/3}\rpa\leq1.
\end{equation}
since otherwise the claim holds true with $K_q=1$. 
Noting $\Theta_X>0$ by Lemma \ref{lemma:theta}, we set 
\[
\varepsilon:=6\Theta_X^{-1/3}\lpa\delta_{n,1}^{1/6}+\delta_{n,2}(q)^{1/3}\rpa.
\]
Then it holds that $\varepsilon\geq6D_n(\log d)^{1-1/q}/n^{1/2-1/q}.$ 
In fact, 
\begin{align*}
\frac{n^{1/2-1/q}\varepsilon}{6D_n(\log d)^{1-1/q}}
&\geq\Theta_X^{-1/3}\frac{n^{1/2-1/q}}{D_n(\log d)^{1-1/q}}\cdot\frac{D_n^{2/3}(\log d)^{2/3-2/(3q)}}{n^{1/3-2/(3q)}}\\
&=\Theta_X^{-1/3}\frac{n^{1/6-1/(3q)}}{D_n^{1/3}(\log d)^{1/3-1/(3q)}}
=\lpa\Theta_X^{2/3}\delta_{n,2}(q)^{1/3}\rpa^{-1/2},
\end{align*}
so \eqref{eq:wlog-mom} implies the desired inequality. 
Now, the remaining proof is almost the same as that of Theorem \ref{thm:main}\ref{thm:psi}, where we use Lemma \ref{lemma:mom} instead of Lemma \ref{lemma:psi}. 
\end{proof}

\subsection{Proof of Lemma \ref{lemma:factor}}

First, since $\zeta$ and $Z$ are independent, we have
\begin{align*}
\sup_{y\in\mathbb R^d}P\lpa0\leq \max_{1\leq j\leq d}(Z_j+a_j\zeta-y_j)\leq\varepsilon\rpa
&\leq\ex{\sup_{y\in\mathbb R^d}P\lpa0\leq \max_{1\leq j\leq d}(Z_j+a_j\zeta-y_j)\leq\varepsilon\mid Z\rpa}\\
&\leq\sup_{y\in\mathbb R^d}P\lpa0\leq \max_{1\leq j\leq d}(a_j\zeta-y_j)\leq\varepsilon\rpa.
\end{align*}
Next, let $\mathcal J_+:=\{j:a_j>0\}$ and $\mathcal J_-:=\{j:a_j<0\}$. Then
\begin{align*}
&\sup_{y\in\mathbb R^d}P\lpa0\leq \max_{1\leq j\leq d}(a_j\zeta-y_j)\leq\varepsilon\rpa\\
&\leq\sup_{y\in\mathbb R^d}P\lpa0\leq \max_{j\in\mathcal J_+}(a_j\zeta-y_j)\leq\varepsilon\rpa
+\sup_{y\in\mathbb R^d}P\lpa0\leq \max_{j\in\mathcal J_-}(a_j\zeta-y_j)\leq\varepsilon\rpa\\
&\leq\sup_{y\in\mathbb R^d}P\lpa0\leq \max_{j\in\mathcal J_+}(\zeta-y_j)\leq\varepsilon/\mathfrak a\rpa
+\sup_{y\in\mathbb R^d}P\lpa0\leq \max_{j\in\mathcal J_-}(y_j-\zeta)\leq\varepsilon/\mathfrak a\rpa\\
&\leq\sup_{t\in\mathbb R}P\lpa0\leq \zeta-t\leq\varepsilon/\mathfrak a\rpa
+\sup_{t\in\mathbb R}P\lpa0\leq t-\zeta\leq\varepsilon/\mathfrak a\rpa
\leq\frac{2\varepsilon}{\sqrt{2\pi}\mf{a}}.
\end{align*}
This yields the desired result.\hfill\qed

\subsection{Proof of Proposition \ref{opt-kolmogorov}}

It suffices to show that there is a sequence $(x_n)_{n=1}^\infty$ of real numbers such that
\[
\rho:=\limsup_{n\to\infty}\left|P\lpa\max_{1\leq j\leq d_n}\frac{1}{\sqrt{n}}\sum_{i=1}^n\xi_{ij}\leq x_n\rpa-P\lpa\max_{1\leq j\leq d_n}\zeta_j\leq x_n\rpa\right|>0.
\]

The proof of this result is a slight refinement of the arguments in \cite[Remark 1]{Chen2017}. 
First, by Theorem 1 in \cite[Chapter VIII]{Petrov1975} (see also Eq.(2.41) in \cite[Chapter VIII]{Petrov1975}), if a sequence $x_n\geq0$ satisfies $x_n=O(n^{1/6})$ as $n\to\infty$, we have
\begin{equation}\label{eq:mdp}
\frac{P\lpa\frac{1}{\sqrt{n}}\sum_{i=1}^n\xi_{i1}>x_n\rpa}{1-\Phi(x_n)}=\exp\lpa\frac{\gamma}{6\sqrt{n}}x_n^3\rpa+O\lpa\frac{x_n+1}{\sqrt{n}}\rpa
\end{equation}
as $n\to\infty$, where $\Phi$ denotes the cumulative distribution function of the standard normal distribution. 

Next, for every $n$, we define $x_n\in\mathbb{R}$ as the solution of the equation $\Phi(x)^{d_n}=e^{-1}$, i.e.~$x_n:=\Phi^{-1}(e^{-1/d_n})$. Then we have $x_n-\sqrt{2\log d_n}=o(1/\sqrt{2\log d_n})$ as $n\to\infty$. 
To see this, we set
\[
b_n:=\sqrt{2\log d_n}-\frac{\log\log d_n+\log4\pi}{2\sqrt{2\log d_n}}.
\]
Then, it is well known (e.g.~\cite[Eq.(3.40)]{EKM1997}) that $P(\sqrt{2\log d_n}(\max_{1\leq j\leq d_n}\zeta_j-x_n)\leq t)\to\Lambda(t)$ as $n\to\infty$ for every $t\in\mathbb{R}$, where $\Lambda(t):=\exp(-e^{-t})$. Moreover, since $\Lambda$ is continuous, by \cite[Lemma 2.11]{Vaart1998} we indeed have 
\[
\lim_{n\to\infty}\sup_{t\in\mathbb{R}}\labs P\lpa\sqrt{2\log d_n}\lpa\max_{1\leq j\leq d_n}\zeta_j-b_n\rpa\leq t\rpa-\Lambda(t)\rabs=0.
\]
Since $\Phi(x_n)^{d_n}=P(\sqrt{2\log d_n}(\max_{1\leq j\leq d_n}\zeta_j-b_n)\leq \sqrt{2\log d_n}(x_n-b_n))$, we obtain $\Lambda(\sqrt{2\log d_n}(x_n-b_n)))\to e^{-1}$ as $n\to\infty$. Since $\Lambda^{-1}(e^{-1})=0$, this implies the desired result. 

Now, since $\lim_{n\to\infty}x_n/\sqrt{2\log d_n}=1$, we have $x_n/n^{1/6}\to\sqrt{2}c^{1/6}$ as $n\to\infty$ by assumption. Hence, \eqref{eq:mdp} yields
\[
\frac{P\lpa\frac{1}{\sqrt{n}}\sum_{i=1}^n\xi_{i1}>x_n\rpa}{1-\Phi(x_n)}\to\exp\lpa\frac{\gamma\sqrt{2c}}{3}\rpa
\]
as $n\to\infty$. 
Since $\gamma<0$ by assumption, there is a constant $a\in(0,1)$ such that
\[
\frac{P\lpa\frac{1}{\sqrt{n}}\sum_{i=1}^n\xi_{i1}>x_n\rpa}{1-\Phi(x_n)}\leq a
\]
for sufficiently large $n$. 
For such an $n$, we obtain
\begin{align*}
P\lpa\frac{1}{\sqrt{n}}\sum_{i=1}^n\xi_{i1}\leq x_n\rpa
&\geq1-a(1-\Phi(x_n))
=\Phi(x_n)+\lpa1-a\rpa(1-\Phi(x_n)).
\end{align*}
Now we infer that
\begin{align*}
\rho
&=\limsup_{n\to\infty}\left|P\lpa\frac{1}{\sqrt{n}}\sum_{i=1}^n\xi_{i1}\leq x_n\rpa^{d_n}-\Phi(x_n)^{d_n}\right|\\
&\geq\limsup_{n\to\infty}\Phi(x_n)^{d_n}\left\{(1+\lpa1-a\rpa(1-\Phi(x_n))/\Phi(x_n))^{d_n}-1\right\}\\
&\geq\limsup_{n\to\infty}\Phi(x_n)^{d_n}\cdot d_n\lpa1-a\rpa(1-\Phi(x_n))/\Phi(x_n),
\end{align*}
where the last inequality follows from the inequality $(1+t)^{d_n}\geq 1+d_nt$ holding for all $t\geq0$. Since $\Phi$ is bounded by 1, we obtain
\begin{align*}
\rho
&\geq\limsup_{n\to\infty}\Phi(x_n)^{d_n}\cdot d_n\lpa1-a\rpa(1-\Phi(x_n))
=\frac{1-a}{e}\limsup_{n\to\infty}d_n(1-\Phi(x_n)).
\end{align*}
Since $d_n(1-\Phi(x_n))=-(e^{-1/d_n}-1)/(1/d_n)\to1$ as $n\to\infty$, we conclude $\rho\geq(1-a)/e>0.$\hfill\qed

\section{Proofs for Section \ref{sec:boot}}\label{sec:proof3}

Throughout this section, we use the following notation: 
We set $Y=(Y_i)_{i=1}^n:=(X_i-\bar{X})_{i=1}^n$. 
Given a sequence $\xi=(\xi_i)_{i=1}^n$ of random vectors, we set $w\xi:=(w_i\xi_i)_{i=1}^n$. 
Note that we have $S_n^{\mathrm{WB}}=S_n^{wY}$. 



\subsection{Proof of Theorem \ref{thm:wild}\ref{thm:wild-psi}}

We may assume $\Theta_X^{2/3}\lpa(b^2\delta_{n,1})^{1/6}+(b^2\delta_{n,2})^{1/3}\rpa\leq1$ without loss of generality. Set 
\[
\varepsilon:=24\Theta_X^{-1/3}\lpa(b^2\delta_{n,1})^{1/6}+(b^2\delta_{n,2})^{1/3}\rpa
\]
and $\kappa_n:=2B_n\log n$. As in the proof of Theorem \ref{thm:main}\ref{thm:psi}, we can prove $\ve\geq3b\cdot4\kappa_n(\log d)/\sqrt{n}$. 
Now, we define $\wt{X}=(\wt{X}_i)_{i=1}^n$ as in the proof of Lemma \ref{lemma:psi}. 
Then we set $\wt{Y}=(\wt{Y}_i)_{i=1}^n:=(\wt{X}_i-\bar{\wt{X}})_{i=1}^n$. 
Note that $\max_{i,j}|\wt{Y}_{ij}|\leq4\kappa_n$, so we have $\max_{i,j}|w_i\wt{Y}_{ij}|\leq b\cdot4\kappa_n$. 
We apply Lemma \ref{coupling} to $w\wt{Y}^\diamond$ with $\mf{C}=\expe{S_n^{X^\diamond}(S_n^{X^\diamond})^\top}$, conditionally on $X$. Then, we conclude that there is an event $\Omega_0\in\mcl{F}$ such that $P(\Omega_0)=1$ and 
\begin{multline*}
P\lpa\max_{1\leq j\leq 2d}\lpa S_{n,j}^{w\wt{Y}^\diamond}-y_j\rpa\in A\mid X\rpa
\leq P\lpa\max_{1\leq j\leq 2d}(Z_{n,j}^{X^\diamond}-y_j)\in A^{5\varepsilon}\rpa\\
+C\varepsilon^{-2}\lpa\Delta^*_{n,0}\log d+\Delta^*_{n,1}\sqrt{\frac{(\log d)^3}{n}}\rpa
\quad\text{on }\Omega_0
\end{multline*}
for any $y\in\mathbb{R}^{2d}$ and $A\in\mcl{B}(\mathbb{R})$, where $C>0$ is a universal constant and
\begin{align*}
\Delta^*_{n,0}&:=\ex{\max_{1\leq j,k\leq d}\labs\frac{1}{n}\sum_{i=1}^n\lpa w_i^2\wt{Y}_{ij}\wt{Y}_{ik}-\expe{X_{ij}X_{ij}}\rpa\rabs\mid X},\quad
\Delta^*_{n,1}:=\sqrt{\frac{1}{n}\ex{\max_{1\leq j\leq d}\sum_{i=1}^nw_i^4\wt{Y}_{ij}^4\mid X}}.
\end{align*}
Thus we obtain
\begin{multline}\label{coupling-boot}
P\lpa\max_{1\leq j\leq 2d}\lpa \lpa S_{n,j}^{\mathrm{WB}}\rpa^\diamond-y_j\rpa\in A\mid X\rpa
\leq P\lpa\max_{1\leq j\leq 2d}(Z_{n,j}^{X^\diamond}-y_j)\in A^{6\varepsilon}\rpa\\
+P\lpa\|S_n^{w(Y^\diamond-\wt{Y}^\diamond)}\|_{\ell_\infty}>\varepsilon\mid X\rpa
+C\varepsilon^{-2}\lpa\Delta^*_{n,0}\log d+\Delta^*_{n,1}\sqrt{\frac{(\log d)^3}{n}}\rpa
\quad\text{on }\Omega_0
\end{multline}
for any $y\in\mathbb{R}^{2d}$ and $A\in\mcl{B}(\mathbb{R})$. 
Now, noting that $w_i$'s are bounded by $b$ and $\sqrt{n}(\bar{X}-\bar{\wt{X}})=S_n^{X-\wt{X}}$, the same argument as in the proof of \eqref{est:trunc} yields
\[
P\lpa\|S_n^{w(Y^\diamond-\wt{Y}^\diamond)}\|_{\ell_\infty}>\varepsilon\rpa\lesssim \varepsilon^{-2}\frac{b^2B_n^2(\log d)^2(\log n)^2}{n}.
\]
Meanwhile, Lemmas \ref{lemma:nemirovski} and \ref{lemma:bounded} and the inequality $\expe{w_i^4}\leq b^2\expe{w_i^2}=b^2$ imply that
\begin{align*}
\ex{\max_{1\leq j,k\leq d}\labs\frac{1}{n}\sum_{i=1}^n\lpa w_i^2\wt{Y}_{ij}\wt{Y}_{ik}-\ex{w_i^2\wt{X}_{ij}\wt{X}_{ij}}\rpa\rabs}
&\lesssim \sqrt{\frac{\log d}{n}}\Delta_{n,1}^{w\wt{X}}
+\ex{\|\bar{\wt{X}}\|_{\ell_\infty}^2}\ex{\left|\frac{1}{n}\sum_{i=1}^nw_i^2\right|}\\
&\lesssim \sqrt{\frac{b^2B_n^2\log d}{n}}+\frac{b^2\kappa_n^2\log d}{n}
+\ex{\|\bar{\wt{X}}\|_{\ell_\infty}^2}.
\end{align*}
Also, $\ex{\|\bar{\wt{X}}\|_{\ell_\infty}^2}\lesssim\kappa_n^2(\log d)/n$ by Lemma 14.14 in \cite{BvdG2011}. 
Then, since $\expe{w_i^2\wt{X}_{ij}\wt{X}_{ij}}=\expe{\wt{X}_{ij}\wt{X}_{ij}}$ and $b\geq1$, the above inequalities and \eqref{est:Delta0} yield
\[
\expe{\Delta^*_{n,0}}\lesssim \sqrt{\frac{b^2B_n^2\log d}{n}}+\frac{b^2B_n^2(\log d)(\log n)^2}{n}.
\]
Moreover, since the Jensen inequality yields $\bar{\wt{X}}_j^4\leq n^{-1}\sum_{i=1}^n\wt{X}_{ij}^4$, we have $\expe{\Delta_{n,1}^*}\lesssim\Delta_{n,1}^{w\wt{X}}+b\Delta_{n,1}^{\wt{X}}$ by the Lyapunov inequality. Thus, Lemma \ref{lemma:bounded} implies that
\[
\expe{\Delta^*_{n,1}}\sqrt{\frac{(\log d)^3}{n}}\lesssim \sqrt{\frac{b^2B_n^2\log d}{n}}+\frac{b^2B_n^2(\log d)(\log n)^2}{n}.
\]
Combining these estimates with Lemma \ref{cck-kolmogorov}, we obtain
\[
\ex{\rho_n^{\mathrm{WB}}(\mcl{A}^{\mathrm{re}}(d))}
\leq2\mcl{C}_{Z_n^X}(6\varepsilon)+C_1\varepsilon^{-2}\lpa\sqrt{b^2\delta_{n,1}}+b^2\delta_{n,2}\rpa,
\]
where $C_1>0$ is a universal constant. Since $\mcl{C}_{Z_n^X}(6\varepsilon)\leq6\Theta_X\ve$ by definition, we obtain the desired result by the definition of $\ve$.\hfill\qed

\subsection{Proof of Theorem \ref{thm:wild}\ref{thm:wild-mom}}

The proof is completely parallel to that of Theorem \ref{thm:wild}\ref{thm:wild-psi}, where we suitably modify the definitions of $\ve,\kappa_n,\wt{X}$ and consider \eqref{est:trunc-mom}--\eqref{est:Delta0-mom} instead of \eqref{est:trunc}--\eqref{est:Delta0}, respectively. 
The detail is omitted. \hfill\qed

\appendix
\section{Appendix: Proofs for Proposition \ref{s-coupling} and Corollary \ref{weak-GA} via the standard Lindeberg method}\label{sec:appendix}

The following lemma is a counterpart of Lemma \ref{lemma:lindeberg} but based on the non-randomized Lindeberg method. 
\begin{lemma}\label{lemma:s-lindeberg}
Under the assumptions of Lemma \ref{lemma:lindeberg}, we have
\begin{align*}
\rho_{h,\beta}(S_n^X,S_n^Y)
\leq C'_mn^{-\frac{m}{2}}\lpa\max_{1\leq l\leq m}\beta^{m-l}\|h^{(l)}\|_\infty\rpa\left\{
\sum_{i=1}^n\expe{\|X_{i}\|_{\ell_\infty}^m}
+\sum_{i=1}^n\expe{\|Y_{i}\|_{\ell_\infty}^m}
\right\},
\end{align*}
where $C'_m>0$ depends only on $m$. 
\end{lemma}

\begin{proof}
We may assume that $X$ and $Y$ are independent without loss of generality. 
Take a vector $y\in\mathbb{R}^d$ and define the function $\Psi:\mathbb{R}^d\to\mathbb{R}$ by $\Psi(x)=h(\Phi_\beta(x-y))$ for $x\in\mathbb{R}^d$. 
For every $k\in[n]$, we set 
\[
S_n(k):=\frac{1}{\sqrt{n}}\sum_{i=1}^kX_{i}+\frac{1}{\sqrt{n}}\sum_{i=k+1}^nY_{i}\quad
\text{and}
\quad
\wh{S}_n(k):=S_n(k)-\frac{1}{\sqrt{n}}X_{k}.
\]
By construction $\wh{S}_n(k)$ is independent of $X_{k}$ and $Y_{k}$. 
We also have $S_n(k)=\wh{S}_n(k)+n^{-1/2}X_{k}$ and $S_n(k-1)=\wh{S}_n(k)+n^{-1/2}Y_{k}$ (with $S_n(0):=n^{-1/2}\sum_{i=1}^nY_{i}$). 
Moreover, it holds that $S_n(n)=S_n^X$ and $S_n(0)=S_n^Y$. 
Therefore, we have
\begin{align}
\left|\ex{\Psi(S_n^X)}-\ex{\Psi(S_n^Y)}\right|
&\leq\sum_{k=1}^n\left|\ex{\Psi(S_n(k))}-\ex{\Psi(S_n(k-1))}\right|.\label{s-lindeberg-eq1}
\end{align}
Meanwhile, when $W=X$ or $W=Y$, Taylor's theorem and the independence of $W_{k}$ from $\wh{S}_n(k)$ yield
\[
\ex{\Psi\lpa\wh{S}_n(k)+n^{-1/2}W_{k}\rpa}
=\sum_{\lambda\in\mathbb{Z}_+^d:|\lambda|\leq m-1}\frac{n^{-|\lambda|/2}}{\lambda!}\ex{\partial^\lambda\Psi\lpa\wh{S}_n(k)\rpa}\ex{W_{k}^{\lambda}}
+R_k[W],
\]
where
\[
R_k[W]:=n^{-m/2}\sum_{\lambda\in\mathbb{Z}_+^d:|\lambda|=m}\frac{m}{\lambda!}\int_0^1(1-t)^{m-1}\ex{\partial^\lambda\Psi\lpa\wh{S}_n(k)+tn^{-1/2}W_{k}\rpa W_{k}^\lambda}dt.
\]
Since we have $\expe{X_{i}^\lambda}=\expe{Y_{i}^\lambda}$ for all $i\in[N]$ and $\lambda\in\mathbb{Z}_+^d$ with $|\lambda|\leq m-1$ by assumption, we obtain
\begin{align}
\labs\ex{\Psi\lpa S_n(k)\rpa}-\ex{\Psi\lpa S(k-1)\rpa}\rabs
\leq|R_k[X]|+|R_k[Y]|.\label{s-lindeberg-eq2}
\end{align}
Now, for any random vector $W$ in $\mathbb{R}^d$ we have by Lemma \ref{cck-derivative}
\begin{align*}
|R_k[W]|\leq c'_mn^{-\frac{m}{2}}\max_{1\leq l\leq m}\beta^{m-l}\|h^{(l)}\|_\infty\sum_{i=1}^n\expe{\|W_{i}\|_{\ell_\infty}^m},
\end{align*}
where $c'_m$ depends only on $m$. Combining this estimate with \eqref{s-lindeberg-eq1} and \eqref{s-lindeberg-eq2}, we obtain the desired result. 
\end{proof}

\begin{proof}[Proof of Proposition \ref{s-coupling}]
Without loss of generality, we may assume 
\begin{equation}\label{eq:s-trivial0}
\varepsilon^{-2}\sqrt{\frac{B_n^4(\log d)^3}{n}}\leq1, 
\end{equation}
since otherwise the claim obviously holds true with $C=1$. 

Set $\beta=\varepsilon^{-1}\log d$ (hence $\beta^{-1}\log d=\varepsilon$). By \eqref{max-smooth} we have
\[
P\lpa\max_{1\leq j\leq d}(S_{n,j}^X-y_j)\in A\rpa\leq P(\Phi_\beta(S_n^X-y)\in A^{\varepsilon})=\expe{1_{A^{\varepsilon}}(\Phi_\beta(S_n^X-y))}.
\]
Next, by Lemma \ref{cck-approx} there is a $C^\infty$ function $h:\mathbb{R}\to\mathbb{R}$ and a universal constant $K>0$ such that $\|h^{(r)}\|_\infty\leq K\varepsilon^{-r}$ for $r=1,2,3,4$ and $1_{A^{\varepsilon}}(x)\leq h(x)\leq1_{A^{4\varepsilon}}(x)$ for all $x\in\mathbb{R}$. Then we have $\expe{1_{A^{\varepsilon}}(\Phi_\beta(S_{n}^X-y))}\leq \expe{h(\Phi_\beta(S_{n}^X-y))}.$ 
Now, let us define $Y=(Y_i)_{i=1}^n$ as in Proposition \ref{prop:beta}. Then we have
\[
\rho_{h,\beta}(S_n^X,S_n^Y)\lesssim\varepsilon^{-4}\frac{B_n^4(\log d)^3}{n}
\]
by Lemma \ref{lemma:s-lindeberg}. Combining this with Proposition \ref{prop:beta}, we obtain
\begin{align*}
\rho_{h,\beta}(S_n^X,Z)
&\leq \rho_{h,\beta}(S_n^X,S_n^Y)+\rho_{h,\beta}(S_n^Y,Z)\\
&\lesssim \varepsilon^{-2}(\log d)\lpa\Delta^X_{n,0}+\Delta^X_{n,1}\sqrt{\frac{\log d}{n}}\rpa
+\varepsilon^{-4}\frac{B_n^4(\log d)^3}{n}\\
&\leq \varepsilon^{-2}(\log d)\Delta^X_{n,0}+\varepsilon^{-2}B_n^2\sqrt{\frac{(\log d)^3}{n}}
+\varepsilon^{-4}\frac{B_n^4(\log d)^3}{n}\\
&\leq \varepsilon^{-2}(\log d)\Delta^X_{n,0}+2\varepsilon^{-2}\sqrt{\frac{B_n^4(\log d)^3}{n}},
\end{align*}
where the last inequality follows from \eqref{eq:s-trivial0}. 
Meanwhile, we also have 
\[
\expe{h(\Phi_\beta(Z-y))}
\leq \expe{1_{A^{4\varepsilon}}(\Phi_\beta(Z-y))}
\leq \ex{1_{A^{5\varepsilon}}\lpa\max_{1\leq j\leq d}(Z_j-y_j)\rpa}=P\lpa\max_{1\leq j\leq d}(Z_j-y_j)\in A^{5\varepsilon}\rpa.
\]
Consequently, we complete the proof.
\end{proof}

\begin{proof}[Proof of Corollary \ref{weak-GA}]
Set $\ve:=\Theta_X^{-1/3}(n^{-1}\log^3d)^{1/6}$. 
Applying Proposition \ref{s-coupling} to $X^\diamond$, we obtain
\begin{align*}
P\lpa\max_{1\leq j\leq 2d}(S_{n,j}^{X^\diamond}-y_j)\in A\rpa
\leq P\lpa\max_{1\leq j\leq 2d}(Z_{j}^{X^\diamond}-y_j)\in A^{6\varepsilon}\rpa
+C_0\varepsilon^{-2}\sqrt{\frac{B_n^4\log^3d}{n}},
\end{align*}
where $C_0>0$ is a universal constant. 
Since $S_n^{X^\diamond}=(S_n^X)^\diamond$, we obtain by Lemma \ref{cck-kolmogorov} 
\begin{align*}
\rho_n(\mcl{A}^{\mathrm{re}})
&\leq2\mcl{C}_{Z_n^X}(6\varepsilon)+C_0\varepsilon^{-2}\sqrt{\frac{\log^3d}{n}}
\lesssim\Theta_X^{2/3}\lpa\frac{B_n^4\log^3d}{n}\rpa^{1/6}.
\end{align*}
This proves the first inequality in Corollary \ref{weak-GA}. The second one follows from Lemma \ref{lemma:nazarov}. 
\end{proof}

\section*{Acknowledgements}

The author thanks Denis Chetverikov and Arun Kumar Kuchibhotla for their valuable comments. 
The author is also grateful to two anonymous referees for their constructive comments which have significantly improved a former version of this paper.  
This work was supported by JST CREST Grant Number JPMJCR14D7 and JSPS KAKENHI Grant Numbers JP17H01100, JP18H00836, JP19K13668.

\section*{Conflict of Interest}

The author declares that there is no conflict of interest.

{\small
\renewcommand*{\baselinestretch}{1}\selectfont
\addcontentsline{toc}{section}{References}

}

\end{document}